\documentclass{amsart}
\usepackage[margin=1in]{geometry}
\usepackage{graphicx} 
\usepackage{amsthm,xcolor}
\usepackage[colorlinks=true,citecolor=black,linkcolor=black,urlcolor=blue]{hyperref}

\usepackage{blkarray}
\usepackage{url}
\usepackage{float}
\usepackage{cleveref}
\usepackage{tikz}
\usepackage{tikz-cd}
\usepackage{mathrsfs}

\usepackage[en-US]{datetime2}
\DTMlangsetup{showdayofmonth=false}

\definecolor{c1}{HTML}{0085FF}
\definecolor{c2}{HTML}{FF9000}
\definecolor{c3}{HTML}{90D300}
\definecolor{c4}{HTML}{9728A1}

\def\PP{\mathcal{P}}
\def\RR{\mathcal{R}}
\def\bfd{{\mathbf{d}}}
\def\bfr{{\mathbf{r}}}

\usepackage{amsmath}
\usepackage{amssymb}
\usepackage{mathtools}
\newtheorem{theorem}{Theorem}[section]

\newtheorem{proposition}[theorem]{Proposition}
\newtheorem{corollary}[theorem]{Corollary}
\newtheorem{lemma}[theorem]{Lemma}
\theoremstyle{definition}
\newtheorem{definition}[theorem]{Definition}
\newtheorem{example}[theorem]{Example}
\newtheorem{remark}[theorem]{Remark}

\everymath{\displaystyle}

\definecolor{SJUBlue}{RGB}{0, 160, 230}
\definecolor{bluegreen}{RGB}{144,144,255}
\definecolor{darkgreen}{RGB}{31, 145, 61}

\newcommand{\crit}{\mathrm{\crit}}

\newcommand{\precdot}{\prec\mathrel{\mkern-5mu}\mathrel{\cdot}}
\newcommand{\leqt}{\preceq_{\text{T}}}
\newcommand{\geqt}{\succeq_{\text{T}}}
\newcommand{\leqd}{\preceq_{\text{D}}}
\newcommand{\lessdott}{\precdot_{\text{T}}}
\newcommand{\lessdotd}{\precdot_{\text{D}}}

\newcommand{\delete}[1]{}

\title{Varieties of Chain Complexes and Mixed Dimer Covers}
\author{}
\date{\today}

\author[Anderson]{Portia X. Anderson}
\address[P.~X. Anderson]{Department of Mathematics, Cornell University, Ithaca, NY 14853}
\email{\textcolor{blue}{\href{mailto:pxa2@cornell.edu}{pxa2@cornell.edu}}}

\author[Banaian]{Esther Banaian}
\address[E.~Banaian]{Institute for Mathematics, Paderborn University, Paderborn, Germany}
\email{\textcolor{blue}{\href{mailto:Esther.Banaian@math.uni-paderborn.de}{Esther.Banaian@math.uni-paderborn.de}}}

\author[Ferreri]{Melanie J. Ferreri}
\address[M. ~J. Ferreri]{Department of Mathematics, College of William \& Mary, Williamsburg, VA, 23187}
\email{\textcolor{blue}{\href{mailto:mjferreri@wm.edu}{mjferreri@wm.edu}}}

\author[Mayers]{Nicholas Mayers}
\address[N.~Mayers]{Department of Mathematics, Kennesaw State University, Marietta, GA, 30060}
\email{\textcolor{blue}{\href{mailto:nmayers@ncsu.edu}{nmayers@ncsu.edu}}}

\author[Wang]{Shiyun Wang}
\address[S.~Wang]{Department of Mathematics, University of Minnesota Twin Cities, MN, 55455}
\email{\textcolor{blue}{\href{mailto:wang8406@umn.edu}{wang8406@umn.edu}}}

\author[Wilson]{Alexander N.~Wilson}
\address[A.~N.~Wilson]{Department of Mathematics, York University, Toronto, ON, Canada}
\email{\textcolor{blue}{\href{mailto:math@alexandernwilson.com}{math@alexandernwilson.com}}}

\begin{document}

\maketitle
\begin{abstract}
    A quiver representation consists of a collection of vector spaces along with a set of arrows, which are linear maps between these spaces. In this work, we study quiver representations in equioriented type $A$ which are also chain complexes; that is, in which consecutive arrows compose to zero. We show that orbits of these representations under a change of basis action are in bijection with mixed dimer covers of a $2 \times n$ grid graph. The latter object can be endowed with a partial order which is a distributive lattice, and we show that the degeneration order on the orbits of chain complexes is a coarsening of this partial order. In addition, we use recent matrix formulae of Claussen and Ovenhouse to enumerate these orbits. This also computes the Kostant partition function  applied to height-restricted, type $A$ roots. When the dimension vector is uniform, we discuss a correspondence with paths of a beam of light bouncing between glass plates and give an explicit generating function.
\end{abstract}

\section{Introduction}
Gabriel's celebrated theorem implies that orbits of quiver representations of types $A$, $D$, and $E$ are in bijection with partitions of their dimension vector into roots of the associated Lie algebra. Such partitions are enumerated by the \emph{Kostant partition function}.  Classically, for a semisimple Lie algebra, the values of the Kostant partition function are used in computing the multiplicity of a weight of an irreducible representation. It is notoriously difficult to find closed formulas for these values. Nevertheless,  this enumeration problem is well studied and has been connected to several other combinatorial objects, including 
flow polytopes \cite{baldonivergne2008}, juggling sequences \cite{benedetti2020kostant}, 
and Tesler matrices \cite{dmtcs:2475}. 

The orbits of representations of a quiver with a fixed dimension vector admit a partial order by containment of closures known as the \emph{degeneration order}. 
The degeneration order is a classical object of study, see for example 
\cite{abeasis1985degenerationsD, abeasis1985degenerations, Riedtmann}. 
However, there is little known about the combinatorial structure of this partial order. 

This work began with a hope to shed more light on these classical problems in type $A$. In the process, we discovered much nicer behavior when we restricted our set of positive roots to those of height at most two. This corresponds to restricting to varieties of \emph{complexes}, that is, orbits of representations whose maps all compose to 0. Equivalently, this corresponds to studying modules of the quotient of the path algebra of a linearly oriented type $A$ quiver modulo the square of the Jacobson radical. This algebra is an instance of a gentle algebra and appears prominently in \cite{Schemes}.  

Our first observation was a straightforward bijection between orbits of complexes with a fixed dimension vector (equivalently, Kostant partitions using roots of heights one and two) and \emph{mixed-dimer covers} of a grid graph. There has been a surge of work lately on (mixed-)dimer covers of graphs, see for example \cite{ClaussenOvenhouse,zbMATH08171225,zbMATH08171225,zbMATH07567672}. In particular, in \cite{ClaussenOvenhouse}, Claussen and Ovenhouse provide an enumeration formula via matrix products for mixed dimers on a family of graphs which includes the grid graph. Therefore, this bijection provides an enumeration of the type $A$, height-restricted Kostant partitions of interest here as a corollary; see Corollary \ref{cor:Enumeration}. In the special case of a dimension vector $\mathbf{d} = (1,2,\ldots, n)$, we see that the number of Kostant partitions with roots of height at most two is an Euler number. These results also have a natural $q$-analogue, provided in Proposition \ref{prop:qAnalog}.

Mixed dimer covers of the grid graph also carry a partial order, called here the ``twist order''. The story of this partial order began with Propp's investigation of single dimers \cite{propp2002lattice}. It was further studied for $d$-dimer covers in \cite{zbMATH08171225} and certain mixed dimer covers in \cite{claussen2020expansion}. In each case, the partial order is remarkably a distributive lattice, and understanding the underlying poset of join-irreducibles is helpful for calculations concerning surface-type cluster algebras \cite{musiker2013bases}. In Theorem \ref{thm:dlattice}, we generalize the work in \cite{claussen2020expansion} by explicitly describing the poset of join-irreducibles for any lattice of mixed-dimer covers of a $2 \times n$ grid graph. 

Now, the first bijection between orbits of complexes and mixed dimer covers noted above dramatically fails to respect either partial order. However, in Theorem \ref{thm:Phiorderpreserving}, we provide a second bijection which is order-preserving in one direction. In particular, this bijection shows that the twist order is a refinement of the degeneration order. We remark that we retain the previous, non-order-preserving bijection in this paper as it is simpler and shows more clearly that the enumerations of these sets are closely related.  

This entire narrative depends on a dimension vector $\mathbf{d} = (d_1,\ldots,d_n)$. In Section \ref{sec:uniformvector}, we restrict to the case where all $d_i$ are equal, i.e., the $d$-dimer case. Inspired by a few OEIS entries, we show a correspondence with paths of a beam of light bouncing between glass plates and provide an explicit generating function recording the number of $d$-dimers.  We conclude with directions for future work in \Cref{sec:futurework}.

\section{Preliminaries}\label{sec:prelim}

In this section, we introduce our objects of interest along with relevant background. As each collection comes equipped with a partial order with which we will be concerned, we begin by covering the requisite background from the theory of posets.

\subsection{Posets}

A \emph{poset} $(\mathcal{P},\preceq)$ consists of a set $\mathcal{P}$ along with a binary relation $\preceq$ between the elements of $\mathcal{P}$ that is
reflexive, anti-symmetric, and transitive. 
For $x,y\in\mathcal{P}$, if $x\preceq y$ and $x\neq y$, then we write $x\prec y$. 
In the case that $x\prec y$ and there exists no $z\in \mathcal{P}$ satisfying $x\prec z\prec y$, then $x\prec y$ is a \emph{covering relation}, denoted $x\precdot y$, and we say that $y$ \emph{covers} $x$. The \emph{Hasse diagram} $\mathcal{H}(\mathcal{P})$ of $\mathcal{P}$ consists of the graph whose vertices are the elements of $\mathcal{P}$ where if $p\prec q$ in $\mathcal{P}$ then the vertex $q$ is higher up than $p$, and vertices $p$ and $q$ are connected by an edge whenever $p\precdot q$. For further details concerning posets, we recommend \cite{EC1}.

Given posets $(\mathcal{P},\preceq_{\mathcal{P}})$ and $(\mathcal{Q},\preceq_{\mathcal{Q}})$, a map $f:\mathcal{P}\to\mathcal{Q}$ is called \emph{order-preserving} if $p\preceq_\mathcal{P} q$ implies that $f(p)\preceq_\mathcal{Q} f(q)$. In the case that an order-preserving map $f:\mathcal{P}\to\mathcal{Q}$ has an order-preserving inverse, we say that $f$ is a (poset) \emph{isomorphism} and that $(\mathcal{P},\preceq_{\mathcal{P}})$ and $(\mathcal{Q},\preceq_{\mathcal{Q}})$ are \emph{isomorphic}, denoted $(\mathcal{P},\preceq_{\mathcal{P}})\cong(\mathcal{Q},\preceq_{\mathcal{Q}})$. 

For a poset $(\mathcal{P},\preceq)$, an \emph{order ideal} is a subset $I\subseteq\mathcal{P}$ with the property that if $y\in I$ and $x\preceq y$, then $x\in I$. The empty set is an order ideal of any poset. The collection of order ideals of $(\mathcal{P},\preceq)$ is denoted $\mathcal{J}(\mathcal{P})$. We say that $(\mathcal{P},\preceq)$ is a \emph{lattice} if given any pair of elements $x,y\in\mathcal{P}$, there exists a least upper bound, called a \emph{join} and denoted $x\vee y$, and a greatest lower bound, called a \emph{meet} and denoted $x\wedge y$. In the case that $(\mathcal{P},\preceq)$ is a lattice for which $x\vee(y\wedge z)=(x\vee y)\wedge (x\vee z)$ holds for all $x,y,z\in\mathcal{P}$, then we call $(\mathcal{P},\preceq)$ a \emph{distributive lattice}. Order ideals and distributive lattices are related by the following well-known result which will be used later.

\begin{theorem}[Fundamental Theorem of Finite Distributive Lattices]\label{thm:FTFDL}
    A poset $(L,\preceq_L)$ is a finite distributive lattice if and only if there exists a finite poset $(\mathcal{P},\preceq_{\mathcal{P}})$ for which $(L,\preceq_L)\cong(\mathcal{J}(\mathcal{P}),\subseteq)$.
\end{theorem}

\subsection{Vector Partitions}
\label{subsec:partitions}

 Let $\mathbf{e}_i\in \mathbb{Z}_{\ge 0}^n$ be the vector consisting of all zeros except a one in position $i$. In this paper, central objects of study are vector partitions of some vector $\mathbf{d}=(d_1,\ldots,d_n)\in\mathbb{Z}_{\ge 0}^n$ into vectors $\mathbf{e}_i$ for $1\leq i\leq n$ and $\mathbf{e}_i+\mathbf{e}_{i+1}$ for $1\leq i\leq n-1$.
Denote the set of such vector partitions by $\PP_\bfd$.

\begin{example}\label{ex:PP}
    For $\mathbf{d}=(2,3,1)$, the collection $\mathcal{P}_\mathbf{d}$ consists of the vector partitions
    $$\{\mathbf{e}_1+\mathbf{e}_2,\mathbf{e}_1+\mathbf{e}_2,\mathbf{e}_2+\mathbf{e}_3\},\quad
    \{\mathbf{e}_1,\mathbf{e}_2,\mathbf{e}_1+\mathbf{e}_2,\mathbf{e}_2+\mathbf{e}_3\},\quad
    \{\mathbf{e}_2,\mathbf{e}_3,\mathbf{e}_1+\mathbf{e}_2,\mathbf{e}_1+\mathbf{e}_2\},$$
    $$\{\mathbf{e}_1,\mathbf{e}_1,\mathbf{e}_2,\mathbf{e}_2,\mathbf{e}_2+\mathbf{e}_3\},\quad
    \{\mathbf{e}_1,\mathbf{e}_2,\mathbf{e}_2,\mathbf{e}_3,\mathbf{e}_1+\mathbf{e}_2\},\quad\text{and}\quad
    \{\mathbf{e}_1,\mathbf{e}_1,\mathbf{e}_2,\mathbf{e}_2,\mathbf{e}_2,\mathbf{e}_3\}.$$
\end{example}

An alternative way to view an element $p\in\PP_\bfd$ is as a tuple $\bfr=(r_1,r_2,\ldots,r_{n-1})\in\mathbb{Z}^{n-1}_{\ge 0}$ where $r_i$ is the number of times that the vector $\mathbf{e}_i+\mathbf{e}_{i+1}$ occurs in $p$ for $1\le i\le n-1$. A tuple in $\mathbb{Z}_{\ge 0}^{n-1}$ corresponds to an element of $\PP_\bfd$ exactly when the following conditions are satisfied: \begin{equation}\label{eq:rankconditions}
r_i\leq \min(d_i,d_{i+1}) \quad\text{for all}\quad 1\leq i\leq n-1\qquad\text{and}\qquad r_i+r_{i+1}\leq d_{i+1}\quad\text{for all}\quad 1\leq i\leq n-2.    
\end{equation}
We denote the set of tuples satisfying these conditions by $\RR_\bfd$ and call them \emph{rank tuples}.
This terminology is motivated by the fact that these tuples encode the ranks of the maps
in the orbit of type $A$ quiver representations indexed by the vector partition $p\in\PP_\bfd$ (see Section \ref{sec:complexes}). 

\begin{example}\label{ex:RR}
    For $\mathbf{d}=(2,3,1)$, the collection $\mathcal{R}_\mathbf{d}$ consists of the tuples $$(2,1),\quad (1,1),\quad (2,0),\quad (0,1),\quad (1,0),\quad\text{and}\quad(0,0).$$
\end{example}

Note that we have a natural partial order $\leqd$ on $\RR_\bfd$ defined as follows.

\begin{definition}\label{def:vporder}
Define the relation $\leqd$ on $\RR_\bfd$ for $\mathbf{d}\in\mathbb{Z}^{n-1}_{\ge 0}$ by \[(a_1,a_2,\ldots,a_{n-1})\leqd(b_1,b_2,\ldots,b_{n-1})\]
if and only if $a_i\leq b_i$ for all $1\leq i\leq n-1$. 
\end{definition}

This order coincides with the degeneration order on associated orbits of quiver representations (see Section~\ref{sec:complexes} for further details). The poset $(\mathcal{R}_{\mathbf{d}}, \preceq_D)$ for $\mathbf{d}=(2,3,1)$ is illustrated in Figure~\ref{fig:R231}.

\begin{remark}
When studying the degeneration order of type $A$ quiver representations, as we discuss further in Section \ref{sec:complexes}, it is common to discuss a triangular \emph{rank array}; see for example \cite{abeasis1985degenerations}. Our restricted setting corresponds to rank arrays in which only one diagonal of the triangle is nonzero, hence we only require the rank tuples described here.
\end{remark}

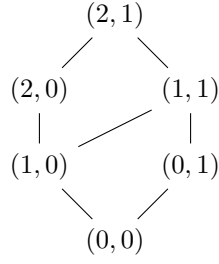
\begin{figure}
    \centering
    \[\begin{tikzpicture}
\node(00) at (0,0){$(0,0)$};
\node(10) at (-1,1){$(1,0)$};
\node(01) at (1,1){$(0,1)$};
\node(20) at (-1,2){$(2,0)$};
\node(11) at (1,2){$(1,1)$};
\node(21) at (0,3){$(2,1)$};
\draw(00) -- (10);
\draw(00) -- (01);
\draw(10) -- (20);
\draw(10) -- (11);
\draw(01) -- (11);
\draw(20) -- (21);
\draw(11) -- (21);
\end{tikzpicture}\]
    \caption{$(\mathcal{R}_{\mathbf{d}}, \preceq_D)$ for $\mathbf{d}=(2,3,1)$}
    \label{fig:R231}
\end{figure}

\subsection{Mixed Dimer Covers}
\label{subsec:dimer}

Let $G=(V,E)$ be a finite graph and $\mathbf{d}:V\to\mathbb{Z}_{\ge 0}$  a function assigning a non-negative integer to each vertex of $G$. Moreover, for each $v\in V$, let $\mathrm{E}(v)\subseteq E$ denote the collection of edges adjacent to $v$ in $G$. A \textit{$\mathbf{d}$-dimer cover} of $G$ is a map $f:E\to\mathbb{Z}_{\ge 0}$ such that $$\sum_{e\in \mathrm{E}(v)}f(e)=\mathbf{d}(v)\qquad \text{for all}\quad v\in V.$$ Note that we can equivalently think of a $\mathbf{d}$-dimer cover of $G$ as a multiset of edges where the function $f$ corresponds to the multiplicity of an edge. More generally, $\mathbf{d}$-dimer covers are referred to as \textit{mixed dimer covers}.

In this article, we will be primarily concerned with mixed dimer covers of the $2\times n$ grid. That is, mixed dimer covers of the graph $G=(V,E)$ with $V=\{(i,0),(i,1)~:~1\le i\le n\}$ and $E$ consisting of $\{(i,0),(i,1)\}$ for $1\le i\le n$ as well as $\{(i,0),(i+1,0)\}$ and $\{(i,1),(i+1,1)\}$ for $1\le i\le n-1$. See Figure~\ref{fig:2bymgrid} for an illustration of the $2\times 6$ grid. To illustrate mixed dimer covers of a graph $G$, we simply write the multiplicity of an edge in the cover beside the given edge (see Figure~\ref{fig:mixddim}).

\begin{figure}[h]
    \centering
    $$\begin{tikzpicture}
\draw (0,0) -- (1,0) -- (2,0) -- (3,0) -- (4,0) -- (5,0) -- (5,1) -- (4,1) -- (3,1) -- (2,1) -- (1,1) -- (0,1) -- (0,0);
\draw (1,0) -- (1,1);
\draw (2,0) -- (2,1);
\draw (3,0) -- (3,1);
\draw (4,0) -- (4,1);
\end{tikzpicture}$$
    \caption{The $2\times 6$ grid.}
    \label{fig:2bymgrid}
\end{figure}
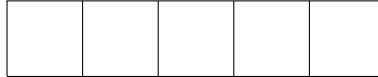

For $\mathbf{d}=(d_1,\hdots,d_n)\in\mathbb{Z}_{\ge 0}^n$, we denote by $\mathcal{D}_\mathbf{d}$ the set of mixed dimer covers of the $2\times n$ grid with the vertices $(i,0)$ and $(i,1)$ assigned the value of $d_i$ for $1\le i\le n$. Such mixed dimer covers were considered in \cite{ClaussenOvenhouse}.

\begin{example}\label{ex:DimerCovers231}
    In Figure~\ref{fig:mixddim}, we illustrate all elements of $\mathcal{D}_\mathbf{d}$ for $\mathbf{d}=(2,3,1)$. Note that such mixed dimer covers are equal in number to the vector partitions in $\mathcal{P}_\mathbf{d}$ illustrated in Example \ref{ex:PP}.
    \begin{figure}[h]
        \centering
        $$\begin{tikzpicture}
\draw(0,0) to node[below]{2} (1,0) to  node[below]{1} (2,0) to  node[right]{0} (2,1) to  node[above]{1} (1,1) to node[above]{2}(0,1) to node[left]{0}(0,0);
\draw (1,0) to node[right]{0} (1,1);
\end{tikzpicture}\quad \begin{tikzpicture}
\draw(0,0) to node[below]{1} (1,0) to  node[below]{1} (2,0) to  node[right]{0} (2,1) to  node[above]{1} (1,1) to node[above]{1}(0,1) to node[left]{1}(0,0);
\draw (1,0) to node[right]{1} (1,1);
\end{tikzpicture}\quad \begin{tikzpicture}
\draw(0,0) to node[below]{2} (1,0) to  node[below]{0} (2,0) to  node[right]{1} (2,1) to  node[above]{0} (1,1) to node[above]{2}(0,1) to node[left]{0}(0,0);
\draw (1,0) to node[right]{1} (1,1);
\end{tikzpicture}$$ $$\begin{tikzpicture}
\draw(0,0) to node[below]{0} (1,0) to  node[below]{1} (2,0) to  node[right]{0} (2,1) to  node[above]{1} (1,1) to node[above]{0}(0,1) to node[left]{2}(0,0);
\draw (1,0) to node[right]{2} (1,1);
\end{tikzpicture}\quad \begin{tikzpicture}
\draw(0,0) to node[below]{1} (1,0) to  node[below]{0} (2,0) to  node[right]{1} (2,1) to  node[above]{0} (1,1) to node[above]{1}(0,1) to node[left]{1}(0,0);
\draw (1,0) to node[right]{2} (1,1);
\end{tikzpicture}\quad \begin{tikzpicture}
\draw(0,0) to node[below]{0} (1,0) to  node[below]{0} (2,0) to  node[right]{1} (2,1) to  node[above]{0} (1,1) to node[above]{0}(0,1) to node[left]{2}(0,0);
\draw (1,0) to node[right]{3} (1,1);
\end{tikzpicture}$$
        \caption{Mixed dimer covers of $\mathcal{D}_\mathbf{d}$ for $\mathbf{d}=(2,3,1)$. The top left dimer cover has multiplicities $h_1 = 2, h_2 = 1, v_1 = 0, v_2 = 0, v_3 = 0$. 
        }
        \label{fig:mixddim}
    \end{figure}
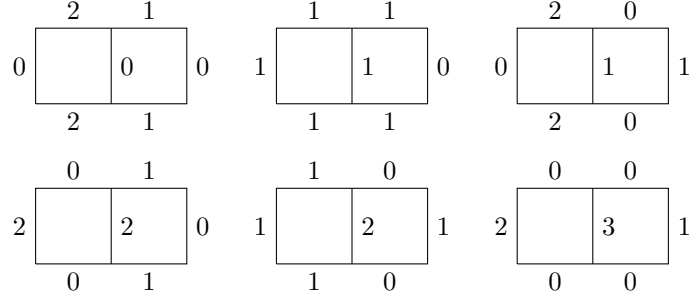
\end{example}

A convenient simplification will be to consider the multiplicity of only certain edges in a mixed dimer cover of a $2 \times n$ grid.

\begin{definition}\label{def:hAndvFunctions}
Given a mixed dimer cover $D\in\mathcal{D}_\mathbf{d}$ for $\mathbf{d}\in\mathbb{Z}_{\ge 0}^n$, let $h_i(D)$ be the multiplicity of $\{(i,0),(i+1,0)\}$ in $D$ for $1\leq i\leq n-1$ and let $v_i(D)$ be the multiplicity of $\{(i,0),(i,1)\}$ in $D$ for $1\leq i\leq n$.
\end{definition}

It is fairly straightforward to see that $h_i(D)$ is also the multiplicity of  $\{(i,1),(i+1,1)\}$ in a mixed dimer cover $D$.

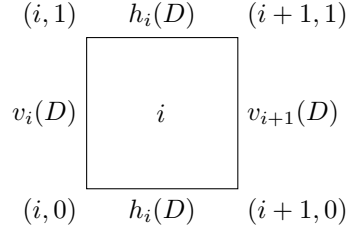
\[\begin{tikzpicture}[scale=1]
  \draw (0,0) rectangle (2,2);

  \node at (0,2) [above left] {$(i,1)$};
  \node at (2,2) [above right] {$(i+1,1)$};
  \node at (0,0) [below left] {$(i,0)$};
  \node at (2,0) [below right] {$(i+1,0)$};

  \node at (1,2) [above] {$h_i(D)$};
  \node at (1,0) [below] {$h_i(D)$};
  \node at (0,1) [left]  {$v_i(D)$};
  \node at (2,1) [right] {$v_{i+1}(D)$};
  \node at (.8,1) [right] {$i$};
\end{tikzpicture}\]

To define a poset structure on $\mathcal{D}_\mathbf{d}$, we identify the cover relations with small local changes in the $h(D)$ and $v(D)$ values. We will naturally index the ``tiles'' $\bigl\{(i,0), (i,1), (i+1,1), (i+1,0)\bigl\}_{i=1}^{n-1}$ of the grid with $1,2,\ldots,n-1$, working from left to right.

\begin{definition}
    Let $D \in \mathcal{D}_{\mathbf{d}}$ with $\mathbf{d}\in\mathbb{Z}^n_{\ge 0}$ be a mixed dimer such that $h_j(D) > 0$ for some $1 \leq j\ \leq n-1$. The \emph{horizontal twist of $D$ at tile $j$} is the mixed dimer $D'\in \mathcal{D}_{\mathbf{d}}$ satisfying $h_i(D') = h_i(D)$ for $i \neq j$, $v_{i}(D') = v_{i}(D)$ for all $i \neq j, j+1$ , $h_j(D') = h_j(D) - 1$, $v_{j}(D') = v_{j}(D) + 1$, $v_{j+1}(D') = v_{j+1}(D) + 1$.
    
    Similarly, let $D \in \mathcal{D}_{\mathbf{d}}$ be a mixed dimer such that $v_j(D) > 0$ and $v_{j+1}(D) > 0$ for some $1 \leq j\ \leq n-1$. The \emph{vertical twist of $D$ at tile $j$} is the mixed dimer $D''\in \mathcal{D}_{\mathbf{d}}$ satisfying $h_i(D'') = h_i(D)$ for $i \neq j$, $v_{i}(D'') = v_{i}(D)$ for all $i \neq j, j+1$, $h_j(D'') = h_j(D) + 1$,  $v_{j}(D'') = v_{j}(D) - 1$, and $v_{j+1}(D'') = v_{j+1}(D) - 1$.
\end{definition}

We may also say that $D'$ (resp. $D''$) is ``a horizontal twist'' (resp. ``a vertical twist'') of $D$ and not specify the tile label. 

\begin{example}
In \Cref{fig:handvtwist}, we have two dimer covers in $\mathcal{D}_{\mathbf{d}}$ with $\mathbf{d} = (3,2,6,4,1,3)$. We can reach the dimer cover on the right from the dimer cover on the left by first performing horizontal twists at tiles 3 and 5 then a vertical twist at tile 4.

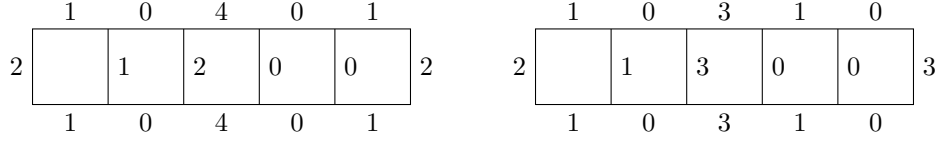
\begin{figure}[h]
    \centering
    \begin{center}
\begin{tikzpicture}
\draw(0,0) to node[below]{1} (1,0) to  node[below]{0} (2,0)to  node[below]{4} (3,0) to  node[below]{0} (4,0) to node[below]{1}(5,0) to node[right]{2}(5,1) to node[above]{1} (4,1) to node[above]{0} (3,1) to node[above]{4} (2,1) to node[above]{0} (1,1) to node[above]{1} (0,1) to node[left]{2}(0,0);
\draw (1,0) to node[right]{1} (1,1);
\draw (2,0) to node[right]{2} (2,1);
\draw (3,0) to node[right]{0} (3,1);
\draw (4,0) to node[right]{0} (4,1);
\end{tikzpicture}
\qquad
\begin{tikzpicture}
\draw(0,0) to node[below]{1} (1,0) to  node[below]{0} (2,0)to  node[below]{3} (3,0) to  node[below]{1} (4,0) to node[below]{0}(5,0) to node[right]{3}(5,1) to node[above]{0} (4,1) to node[above]{1} (3,1) to node[above]{3} (2,1) to node[above]{0} (1,1) to node[above]{1} (0,1) to node[left]{2}(0,0);
\draw (1,0) to node[right]{1} (1,1);
\draw (2,0) to node[right]{3} (2,1);
\draw (3,0) to node[right]{0} (3,1);
\draw (4,0) to node[right]{0} (4,1);
\end{tikzpicture}
\end{center}
    \caption{The dimer cover on the right is the result of twisting the dimer cover on the left horizontally at tiles 3 and 5 then vertically at tile 4.}
    \label{fig:handvtwist}
\end{figure}
\end{example}

Now, we define a partial order $\leqt$ on $\mathcal{D}_{\mathbf{d}}$ in terms of twists as follows.

\begin{definition}\label{def:mdorder}
Define the relation $\leqt$ on $\mathcal{D}_{\mathbf{d}}$ to be the transitive closure of the relation $\lessdott$ where $D \lessdott D'$ if $D'$ can be formed from $D$ by either a vertical twist at an even tile or a horizontal twist at an odd tile.
\end{definition}

A proof that $\leqt$ in fact defines a partial order on $\mathcal{D}_{\mathbf{d}}$ is included in Section~\ref{sec:md}. The poset $(\mathcal{D}_{\mathbf{d}},\preceq_T)$ for $\mathbf{d}=(2,3,1)$ is illustrated in Figure~\ref{fig:D231}.

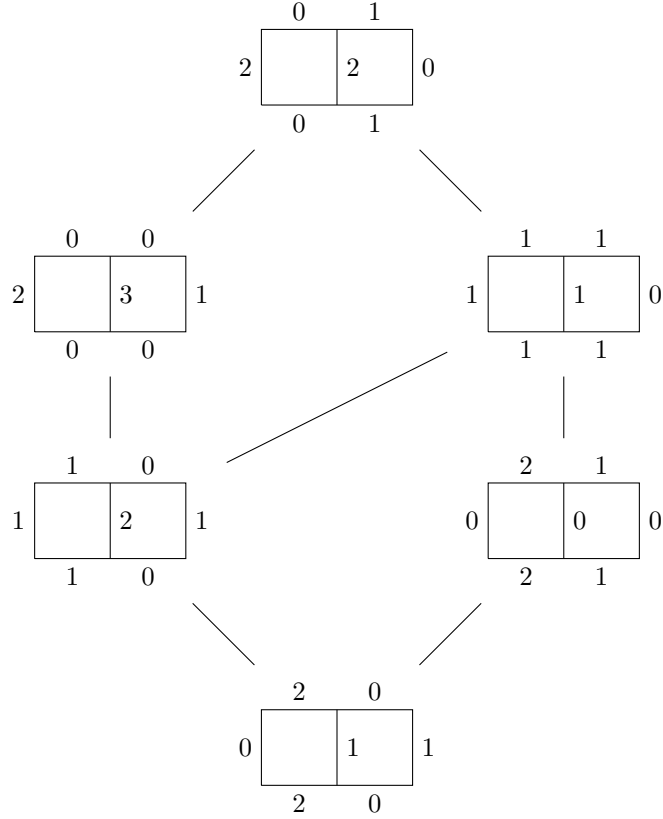
\begin{figure}
    \centering
    $$\begin{tikzpicture}
        \node (1) at (0,0) {\begin{tikzpicture}
\draw(0,0) to node[below]{2} (1,0) to  node[below]{0} (2,0) to  node[right]{1} (2,1) to  node[above]{0} (1,1) to node[above]{2}(0,1) to node[left]{0}(0,0);
\draw (1,0) to node[right]{1} (1,1);
\end{tikzpicture}};
        \node (2) at (-3,3) {\begin{tikzpicture}
\draw(0,0) to node[below]{1} (1,0) to  node[below]{0} (2,0) to  node[right]{1} (2,1) to  node[above]{0} (1,1) to node[above]{1}(0,1) to node[left]{1}(0,0);
\draw (1,0) to node[right]{2} (1,1);
\end{tikzpicture}};
        \node (3) at (3,3) {\begin{tikzpicture}
\draw(0,0) to node[below]{2} (1,0) to  node[below]{1} (2,0) to  node[right]{0} (2,1) to  node[above]{1} (1,1) to node[above]{2}(0,1) to node[left]{0}(0,0);
\draw (1,0) to node[right]{0} (1,1);
\end{tikzpicture}};
        \node (4) at (-3,6) {\begin{tikzpicture}
\draw(0,0) to node[below]{0} (1,0) to  node[below]{0} (2,0) to  node[right]{1} (2,1) to  node[above]{0} (1,1) to node[above]{0}(0,1) to node[left]{2}(0,0);
\draw (1,0) to node[right]{3} (1,1);
\end{tikzpicture}};
        \node (5) at (3,6) {\begin{tikzpicture}
\draw(0,0) to node[below]{1} (1,0) to  node[below]{1} (2,0) to  node[right]{0} (2,1) to  node[above]{1} (1,1) to node[above]{1}(0,1) to node[left]{1}(0,0);
\draw (1,0) to node[right]{1} (1,1);
\end{tikzpicture}};
        \node (6) at  (0,9) {\begin{tikzpicture}
\draw(0,0) to node[below]{0} (1,0) to  node[below]{1} (2,0) to  node[right]{0} (2,1) to  node[above]{1} (1,1) to node[above]{0}(0,1) to node[left]{2}(0,0);
\draw (1,0) to node[right]{2} (1,1);
\end{tikzpicture}};
\draw (1)--(2)--(4)--(6)--(5)--(3)--(1);
\draw (2)--(5);
        
    \end{tikzpicture}$$
    \caption{$(\mathcal{D}_{\mathbf{d}},\preceq_T)$ for $\mathbf{d}=(2,3,1)$}
    \label{fig:D231}
\end{figure}

\subsection{Complexes}\label{sec:complexes}

A \emph{complex} of length $n$ is a sequence of $n$  vector spaces $M_1,M_2,\ldots,M_n$ with $n-1$ linear maps $f_i: M_i \to M_{i+1}$ satisfying $f_{i+1} \circ f_i = 0$ for all $1 \leq i \leq n-2$. Here, we will assume each $M_i$ is a vector space over $\mathbb{C}$.
Complexes are special instances of ``bound quiver representations," which we briefly recall here. For a more complete treatment, we recommend \cite{Assem,schiffler2014quiver}. 

A \emph{quiver} is simply  a directed graph, while a \emph{quiver representation} is an assignment of a vector space to every vertex of the graph and a linear map to every arrow in such a way that the domain of the map coincides with the tail of the arrow and the codomain coincides with the head. The \emph{path algebra} of a quiver $Q$, denoted $\mathbb{C}Q$, has as elements linear combinations of paths in $Q$. Multiplication of two paths is defined by concatenation if possible, and this is extended linearly to define a multiplication for any two elements of $\mathbb{C}Q$. A quiver representation is the same as a (left) $\mathbb{C}Q$ module.

Given an ideal $I$ of $\mathbb{C}Q$, a \emph{bound quiver representation}, also called a $(Q,I)$-representation, is the analogue of a $\mathbb{C}Q/I$-module. In our setting, we only consider monomial ideals. A monomial corresponds to a path in the quiver, and in this setting a bound quiver representation is one such that the composition of linear maps along any monomial in $I$ is equal to the zero map. Therefore, a complex of length $n$ is the same as a $(Q,I)$-representation where $Q$ is a path graph with $n$ vertices and equioriented arrows, and $I$ is generated by all degree 2 monomials. Denote the associated quotient of the path algebra by $C_n$ and denote the quiver and ideal by $(\overrightarrow{A}_n,\mathrm{rad}^2)$.

The set of $(Q,I)$-representations forms an abelian category, and we can thereby consider notions such as the decomposition of a quiver representation into indecomposable representations. Moreover, by the Krull-Schmidt property, these decompositions are unique up to rearrangement. 

It is straightforward to show that, in the setting of complexes, the indecomposable representations consist of the following two sets. Given $1 \leq i \leq n$, let $S_i$ denote the representation with a 1-dimensional vector space at vertex $i$ and 0's elsewhere. Necessarily, all maps in $S_i$ are 0. Next, given $1 \leq i \leq n-1$, let $P_i$ denote the representation with 1-dimensional vector spaces at vertex $i$ and $i+1$, 0's at all other vertices, and an identity map between the two nonzero vector spaces. For example, below we draw $S_2$ and $P_2$ for $(\overrightarrow{A}_3, \mathrm{rad}^2)$.

\[
 0 \xrightarrow{0}  \mathbb{C} \xrightarrow{0} 0 \qquad  0 \xrightarrow{0}  \mathbb{C} \xrightarrow{1} \mathbb{C}
\]

Given a quiver $Q$ with vertices labeled $v_1,\ldots,v_n$, the \emph{dimension vector} of a representation $M$ is $\mathbf{d}(M) = (d_1,\ldots,d_n)$ where $d_i$ is the dimension of the vector space associated to vertex $v_i$.  For example, with $n = 3$ again, $\mathbf{d}(S_2) = (0,1,0)$ and $\mathbf{d}(P_2) = (0,1,1)$.

Let $\mathrm{mod}(C_n,\mathbf{d})$ denote the space of all $(\overrightarrow{A}_n,\mathrm{rad}^2)$-representations with dimension vector $\mathbf{d}$. There is a natural action on  $\mathrm{mod}(C_n,\mathbf{d})$ by \[
\mathrm{GL}_\mathbf{d} := \prod_{i=1}^n \mathrm{GL}_{d_i}(\mathbb{C})
\]
given by change of basis at each vertex. Explicitly, note that a point in  $\mathrm{mod}(C_n,\mathbf{d})$ is determined by the maps $f_i$. An element $\phi = (\phi_i) \in \mathrm{GL}_\mathbf{d}$ sends $(M_i,f_i) \in \mathrm{mod}(C_n,\mathbf{d})$ to $(M_i',f_i')$ where $M_i \cong M_i'$ and $f_i' = \phi_{i+1} f_i \phi_i^{-1}$. 

Given $M \in  \mathrm{mod}(C_n,\mathbf{d})$, define $\mathcal{O}_M$ to be the $\mathrm{GL}_{\mathbf{d}}$ orbit of $M$. We say that two $(\overrightarrow{A}_n,\mathrm{rad}^2)$-representations are \emph{isomorphic} if they are contained in the same orbit $\mathcal{O}_M$. These orbits partition the space $\mathrm{mod}(C_n,\mathbf{d})$, and the \emph{degeneration order}\footnote{Many other sources use an opposite definition of degeneration order; that is, they define $\mathcal{O}_{M'} \preceq \mathcal{O}_{M}$ whenever $\mathcal{O}_M \subseteq \overline{\mathcal{O}_{M'}}$. Our narrative can easily be translated to this setting by reversing conventions. }
 of the orbits of   $\mathrm{mod}(C_n,\mathbf{d})$ is given by setting $\mathcal{O}_M \preceq \mathcal{O}_{M'}$ whenever $\mathcal{O}_M$ is contained in the Zariski closure of $\mathcal{O}_{M'}$, i.e., $\overline{\mathcal{O}_{M'}}$ \cite{Riedtmann}.

 \begin{example}
We exhibit here six pairwise non-isomorphic representations of $(\overrightarrow{A}_3,\mathrm{rad}^2)$ with dimension vector $(2,3,1)$. 

\[
 \mathbb{C}^{2} \xrightarrow{\begin{pmatrix} 1 & 0 \\ 0 & 1 \\ 0 & 0 \end{pmatrix}}  \mathbb{C}^{3} \xrightarrow{\begin{pmatrix} 0 & 0 & 1 \end{pmatrix}} \mathbb{C}  \qquad \mathbb{C}^{2} \xrightarrow{\begin{pmatrix} 1 & 0 \\ 0 & 0 \\ 0 & 0 \end{pmatrix}}  \mathbb{C}^{3} \xrightarrow{\begin{pmatrix} 0 & 1 & 0 \end{pmatrix}} \mathbb{C}  \qquad  \mathbb{C}^{2} \xrightarrow{\begin{pmatrix} 1 & 0 \\ 0 & 1 \\ 0 & 0 \end{pmatrix}}  \mathbb{C}^{3} \xrightarrow{\begin{pmatrix} 0 & 0 & 0 \end{pmatrix}} \mathbb{C} 
\]

\[
 \mathbb{C}^{2} \xrightarrow{\begin{pmatrix} 0 & 0 \\ 0 & 0 \\ 0 & 0 \end{pmatrix}}  \mathbb{C}^{3} \xrightarrow{\begin{pmatrix} 1 & 0 & 0 \end{pmatrix}} \mathbb{C}  \qquad \mathbb{C}^{2} \xrightarrow{\begin{pmatrix} 1 & 0 \\ 0 & 0 \\ 0 & 0 \end{pmatrix}}  \mathbb{C}^{3} \xrightarrow{\begin{pmatrix} 0 & 0 & 0 \end{pmatrix}} \mathbb{C}  \qquad  \mathbb{C}^{2} \xrightarrow{\begin{pmatrix} 0 & 0 \\ 0 & 0 \\ 0 & 0 \end{pmatrix}}  \mathbb{C}^{3} \xrightarrow{\begin{pmatrix} 0 & 0 & 0 \end{pmatrix}} \mathbb{C} 
\]

\end{example}

For a small example, consider $\mod(C_2,(1,1))$, which has orbits $\mathcal{O}_{S_1 \oplus S_2} = \{ \mathbb{C} \xrightarrow{0} \mathbb{C}\}$ and $\mathcal{O}_{P_1} =\{ \mathbb{C} \xrightarrow{a \neq 0} \mathbb{C}\}$. Here, we have $\mathcal{O}_{S_1 \oplus S_2} 
\preceq_{\mathrm{deg}} \mathcal{O}_{P_1}$ because $\mathcal{O}_{S_1 \oplus S_2} \subset \overline{\mathcal{O}_{P_1}}$. 

In general, a $GL_{\mathbf{d}}$ invariant of $\mathrm{mod}(C_n,\mathbf{d})$ is provided by the sequence of ranks $\mathrm{rk}(f_1),\ldots, \mathrm{rk}(f_{n-1})$. Indeed, this data uniquely determines an orbit $\mathcal{O}_M$ since $\mathrm{rk}(f_i)$ is equal to the multiplicity of $P_i$ in the decomposition of any $M' \in \mathcal{O}_M$ into indecomposable representations. 

The fact that $f_i$ is a linear transformation from $\mathbb{C}^{d_i}$ to $\mathbb{C}^{d_{i+1}}$ and $f_{i+1} \circ f_i = 0$ for all $i$ implies that the rank sequence exactly satisfies Equation (\ref{eq:rankconditions}). Conversely, given any element $(r_1,\ldots,r_{n-1}) \in \mathcal{R}_{\mathbf{d}}$, one can construct a module $M \in \mathrm{mod}(C_n, \mathbf{d})$ whose rank data coincides with $(r_1,\ldots,r_{n-1})$. 
Moreover, the degeneration order in this case can be shown to exactly coincide with our partial order on $\mathcal{R}_{\mathbf{d}}$. For the sake of completeness, we summarize this discussion and outline a brief argument.

\begin{proposition}\label{prop:DegenOrder}
Let $\mathbf{d} \in \mathbb{Z}_{\ge 0}^n$. 
\begin{enumerate}
    \item There is a bijection $\Gamma$ from orbits of modules in $(\overrightarrow{A}_n,\mathrm{rad}^2)$ to $\mathcal{R}_{\mathbf{d}}$ given by sending $\mathcal{O}_M$ to $(r_1,\ldots,r_{n-1})$ where $r_i$ is the multiplicity of $P_i$ in the decomposition of any $M' \in \mathcal{O}_M$.
    \item The map $\Gamma$ is a poset isomorphism from the set of orbits of $\mathrm{mod}(C_n,\mathbf{d})$ under $\preceq_{\mathrm{deg}}$ to $(\mathcal{R}_{\mathbf{d}}, \preceq_D)$.
\end{enumerate}
\end{proposition}

\begin{proof}
We have already shown the first statement. For the second statement, by \cite[Corollary 4]{Zwara_2000} it is equivalent to study a different order given by extensions of modules. As described in \cite[Lemma 6.2]{Schemes}, the only nontrivial extensions in $\mathrm{mod}(C_n)$ occur between simple modules supported at adjacent vertices. In particular, for each $1 \leq i < n$, there is a non-split short exact sequence 
\[
0 \to S_{i+1} \to P_i \to S_i \to 0.
\]
This shows that we have cover relations in the degeneration order when we replace a summand $S_i \oplus S_{i+1}$ with $P_i$.  This exactly corresponds to increasing one entry in a rank tuple by 1, yielding a cover relation in  $(\mathcal{R}_{\mathbf{d}}, \preceq_D)$.
\end{proof}

\begin{remark}
The algebra $C_n$ is an instance of a \emph{gentle algebra}. Gentle algebras are currently an area of intense study, due both to the beautiful combinatorial classification of their indecomposable modules \cite{zbMATH03989588,zbMATH03904786} and their connections to other areas of mathematics, such as the theory of cluster algebras \cite{zbMATH05704468} and homological mirror symmetry \cite{zbMATH07159952}. The algebra $C_n$  was used as a type of building block in \cite{Schemes} to study general gentle algebras. The degeneration order of gentle algebras was also recently studied by Marquardt in \cite{marquardt2026}, and Proposition \ref{prop:DegenOrder} could be recovered as a special case of her results. 
\end{remark}

\section{A naive bijection}
\label{sec:naive}

In this section, we give a bijection between mixed dimer covers and rank tuples. We call this bijection ``naive'' because, while it is useful for enumeration, a more complicated bijection is introduced in the following section that better respects the partial orders accompanying the two sets (see Remark~\ref{rem:naive}).

 Recall that for $\mathbf{d} \in \mathbb{Z}_{\geq 0}^n$, the set $\mathcal{D}_{\mathbf{d}}$ is the set of mixed dimer covers of the $2 \times n$ grid with $d_i$ edges incident to vertices $(i,0)$ and $(i,1)$. Throughout the remainder of the paper, set $m_i:= \min(d_i,d_{i+1})$.

The definition of mixed dimer covers on a $2 \times n$ grid makes the following immediate.

\begin{lemma}\label{lem:VsHsDs}
If $D \in \mathcal{D}_{\mathbf{d}}$ for $\mathbf{d}\in\mathbb{Z}_{\ge 0}^n$, then the edge multiplicities in $D$ satisfy  
\begin{equation*}
 h_1(D) + v_1(D) = d_1, \qquad h_{n-1}(D) + v_n(D) = d_n. \label{}
\end{equation*}
and for all $1 \leq i \leq n-2$,
\begin{equation*}
 h_i(D) + h_{i+1}(D) + v_{i+1}(D) = d_{i+1}
\end{equation*}
Moreover, every collection of nonnegative integers $\{h_i\}_{1 \leq i \leq n-1} \cup \{v_i\}_{1 \leq i \leq n}$ satisfying these equations determines an element  $D \in \mathcal{D}_{\mathbf{d}}$. 
\end{lemma}

A more useful result is the following, which says we can focus on a subgraph of the grid graph when considering mixed dimer covers.
 \begin{lemma}\label{lem:HoriztonalDeterminesMatching}
 A mixed dimer cover $D\in \mathcal{D}_{\mathbf{d}}$ for $\mathbf{d}\in\mathbb{Z}_{\ge 0}^n$ is uniquely determined by the tuple $(h_i(D))_{1 \leq i \leq n-1}$ where the values $h_i(D)$ satisfy \[h_1(D) \leq d_1, \qquad h_i(D) + h_{i+1}(D) \leq d_{i+1}\quad \text{for all}\quad 1 \leq i \leq n-2, \qquad \text{and} \qquad h_{n-1}(D) \leq d_n.
 \] Moreover, any tuple $(h_1,\ldots,h_{n-1})$ of nonnegative integers satisfying these inequalities determines an element of $\mathcal{D}_{\mathbf{d}}$.
 \end{lemma}

 \begin{proof}
The fact that the values $h_i(D)$ satisfy the given inequalities is clear from Lemma \ref{lem:VsHsDs}. Next, suppose $D,D' \in \mathcal{D}_{\mathbf{d}}$ satisfy $h_i(D) = h_i(D')$ for all $1 \leq i \leq n-1$. Then, by using each of the $n$ equations in Lemma~\ref{lem:VsHsDs}, we have $v_i(D) = v_i(D')$ for all $i$, and by the final statement of the same lemma we see that $D = D'$. 

Similarly, given a tuple $(h_1,\ldots,h_{n-1}) \in \mathbb{Z}_{\ge 0}^{n-1}$ satisfying these inequalities,
we can find values $(v_1,\ldots,v_n) \in \mathbb{Z}_{\ge 0}^{n-1}$ such that the $v_i$ satisfy the equations in Lemma \ref{lem:VsHsDs}. Hence, we can build a mixed dimer cover $D$ with $h_i(D) = h_i$ for all $1 \leq i \leq n-1$.
 \end{proof}

 \begin{remark}\label{rem:VsAlsoDetermine}
In light of Lemma \ref{lem:VsHsDs}, the values $v_i(D)$ also uniquely determine $D$. However, it is not true that any sequence $(v_1,\ldots,v_n)$ such that $v_i \leq d_i$ will correspond to an element of $\mathcal{D}_{\mathbf{d}}$. For example, if $\mathbf{d} = (1,8,1)$, then every $D \in \mathcal{D}_{\mathbf{d}}$ has $v_2(D) \geq 6$.
 \end{remark}

 It is evident that the tuples in $\mathcal{R}_{\mathbf{d}}$ satisfy the same inequalities as in Lemma \ref{lem:HoriztonalDeterminesMatching} and that any such tuple determines a vector partition in $\mathcal{P}_\mathbf{d}$. Thus, we have the following bijection.

\begin{corollary}\label{cor:naive_bijection}
    The map $\Theta:\mathcal{D}_{\bfd}\to\mathcal{R}_\bfd$ for $\bfd\in\mathbb{Z}^n_{\ge 0}$ given by \[\Theta(D)=(h_1(D),h_2(D),\ldots,h_{n-1}(D))\] is a bijection.
  
\end{corollary}

\begin{remark}\label{rem:naive}
    Let $D_1$, $D_2$, $D_3$, and $D_4$ be the mixed dimer covers of $\mathcal{D}_{\mathbf{d}}$ for $\mathbf{d}=(2,3,1)$ with $\Theta(D_1)=(h_1(D_1),h_2(D_1))=(2,0)$, $\Theta(D_2)=(1,0)$, $\Theta(D_3)=(2,1)$, and $\Theta(D_4)=(0,1)$. From Figure \ref{fig:D231}, we see 
    \begin{itemize}
        \item $D_1\preceq_T D_2$ and $\Theta(D_2)\preceq_D\Theta(D_1)$;
        \item $D_1\preceq_T D_3$ and $\Theta(D_1)\preceq_D\Theta(D_3)$;
        \item $D_2$ and $D_3$ are unrelated with respect to $\preceq_T$, while $\Theta(D_2)\prec_D\Theta(D_3)$; and
        \item $D_1\preceq_T D_4$, while $\Theta(D_1)$ and $\Theta(D_4)$ are unrelated with respect to $\preceq_D$.
    \end{itemize}
    The above examples show that the bijection $\Theta:\mathcal{D}_\mathbf{d}\to\mathcal{R}_\mathbf{d}$ does not respect the partial orders $\preceq_T$ and $\preceq_D$ on either collection. In Section~\ref{sec:opbij}, a bijection is provided that better respects these partial orders.
\end{remark}

Thanks to recent work by Claussen and Ovenhouse \cite{ClaussenOvenhouse}, Corollary \ref{cor:naive_bijection} has fairly immediate enumerative consequences. 

\begin{definition}[Definition 4, \cite{ClaussenOvenhouse}]
    Define $R_{a,b}$ to be the $(a+1) \times (b+1)$ matrix where entry $(i,j)$ is 1 if $i+j \leq b+2$ and otherwise is 0.
\end{definition}
For example, 
    \[
R_{3,1}=\begin{pmatrix} 1 & 1 \\ 1 & 0 \\ 0 & 0 \\ 0 & 0 \end{pmatrix}.
 \]

 These matrices were used to enumerate mixed dimer covers on the $2 \times n$ grid.

 \begin{theorem}[Theorem 1, \cite{ClaussenOvenhouse}]\label{thm:ClaussenOvenhouse}
  Given $\mathbf{d} = (d_1,d_2,\ldots,d_n) \in \mathbb{Z}_{\ge 0}^n$, the cardinality of $\mathcal{D}_{\mathbf{d}}$  is the $(1,1)$-entry of \[
R_{d_1,d_1}R_{d_1,d_2} R_{d_2,d_3} \cdots R_{d_{n-1},d_n}.
    \]  
 \end{theorem}

Combining Theorem \ref{thm:ClaussenOvenhouse} and Corollary \ref{cor:naive_bijection} yields the following.

\begin{corollary}\label{cor:Enumeration}
Given $\mathbf{d} = (d_1,d_2,\ldots,d_n) \in \mathbb{Z}_{\ge 0}^n$, the cardinality of $\mathcal{R}_{\mathbf{d}}$ is the $(1,1)$-entry of 
\[
R_{d_1,d_1}R_{d_1,d_2} R_{d_2,d_3} \cdots R_{d_{n-1},d_n}.
    \]  
\end{corollary}

\begin{example}\label{ex:MatrixProductEnumeration}
Let $\mathbf{d} = (2,3,1)$. We compute the product from Theorem \ref{thm:ClaussenOvenhouse}: \[
R_{2,2}R_{2,3} R_{3,1} = \begin{pmatrix} 1 & 1 & 1 \\ 1 & 1& 0\\ 1 & 0 & 0 \end{pmatrix}\begin{pmatrix} 1 & 1 & 1 & 1 \\ 1 & 1 & 1 & 0 \\ 1 & 1 & 0 & 0 \end{pmatrix} \begin{pmatrix} 1 & 1 \\ 1 & 0 \\ 0 & 0 \\ 0 & 0 \end{pmatrix} = \begin{pmatrix} 1 & 1 & 1 \\ 1 & 1& 0\\ 1 & 0 & 0 \end{pmatrix}\begin{pmatrix} 2 & 1 \\ 2 & 1 \\ 2 & 1 \end{pmatrix} = \begin{pmatrix} 6 & 3\\ 4& 2 \\ 2 & 1 \end{pmatrix}
    \]
From Example \ref{ex:DimerCovers231}, we see that there are indeed six mixed dimer covers in $\mathcal{D}_{(2,3,1)}$, which are in correspondence with the six tuples $\mathcal{R}_{(2,3,1)}$, as exhibited in Example \ref{ex:RR} .
\end{example}

\begin{remark}
Theorem 1 of \cite{ClaussenOvenhouse} in fact gives an explanation to all entries of the matrix product.
Entry $(i,j)$ gives the cardinality of $\mathcal{D}_{\bfd'}$
where $\bfd'=(d_1+1-i,d_2,d_2,\ldots,d_{n-1},d_n+1-j)$. For instance, in Example \ref{ex:MatrixProductEnumeration}, the 4 in entry $(2,1)$ is the cardinality of $\mathcal{D}_{(1,3,1)}$. Equivalently, this is the number of mixed dimer covers $D \in \mathcal{D}_{(2,3,1)}$ with $v_1(D) \geq 1$. Following our combinatorial bijections, this is the cardinality of the subset of vector partitions in $\mathcal{P}_{(2,3,1)}$ where the vector $\mathbf{e}_1$ appears with multiplicity at least one.
\end{remark}

For some representation-theoretic computations 
(e.g. the $q$-analog for Kostant's weight multiplicity formula \cite[Proposition 9.2]{Lusztig}), it is often useful
to enumerate vector partitions by their number of parts.
To this end, \cite[Theorem 2]{ClaussenOvenhouse} can be
specialized to yield the following
proposition. 

Let $V_{a,b}(q)$ be the $(a+1)\times(b+1)$ matrix where
entry $(i,j)$ is $q^{b+1-j}$ if $i+j\leq b+2$ and
otherwise is $0$. For example,
\[V_{2,3}(q)=\begin{pmatrix}
    q^3 & q^2 & q & 1\\
    q^3 & q^2 & q & 0 \\
    q^3 & q^2 & 0 &0
\end{pmatrix}.\]
Let $U_a(q)$ be the square $(a+1)\times(a+1)$ matrix
where entry $(i,j)$ is $q^{a+2-i-j}$ if $i+j\leq a+2$
and otherwise is 0. For example,
\[U_3(q)=\begin{pmatrix}
    q^3 & q^2 & q & 1\\
    q^2 & q & 1 & 0 \\
    q & 1 & 0 & 0\\
    1 & 0 & 0 & 0
\end{pmatrix}.\]

Given a vector partition $p$, let $\ell(p)$ denote the number of parts, i.e., the number of vectors. Recall every element of $\mathcal{R}_\mathbf{d}$ uniquely determines a vector partition in $\mathcal{P}_\mathbf{d}$. Let $\Theta'(D)$ be the vector partition associated to $\Theta(D)$. 

\begin{proposition}\label{prop:qAnalog}
    Given $\mathbf{d} = (d_1,d_2,\ldots,d_n) \in \mathbb{Z}_{\ge 0}^n$, the $q$-analog
    \[\sum_{p\in \PP_\bfd} q^{\ell(p)}\]
    can be computed as the $(1,1)$-entry of
    the product
    \[U_{d_1}(q)V_{d_1,d_2}(q)V_{d_2,d_3}(q)\cdots V_{d_{n-1},d_n}(q).\]
\end{proposition}

\begin{proof}
If we set $a_i=b_i=q$ and $c_i=1$ in \cite[Theorem 2]{ClaussenOvenhouse}, then the coefficient of $q^k$ in the $(1,1)$-entry of the matrix product is the number of mixed dimer covers $D$ satisfying $\sum_i h_i(D) + \sum_i v_i(D) = k$. Notice $v_i(D)$ is equal to the multiplicity of $\mathbf{e}_i$ in $\Theta'(D)$ and $h_i(D)$ is equal to the multiplicity of $\mathbf{e}_i+\mathbf{e}_{i+1}$ in $\Theta'(D)$. Therefore, the number $k$ also counts the total number of vectors in $\Theta'(D)$.
\end{proof}

\begin{example}\label{ex:qMatrixProductEnumeration}
Let $\mathbf{d} = (2,3,1)$. We compute the product from Proposition \ref{prop:qAnalog}. \[
U_{2}(q)V_{2,3}(q) V_{3,1}(q) =
\begin{pmatrix}
    q^2 & q & 1 \\
    q & 1& 0\\
    1 & 0 & 0
\end{pmatrix}
\begin{pmatrix}
    q^3 & q^2 & q & 1 \\
    q^3 & q^2 & q & 0 \\
    q^3 & q^2 & 0 & 0
\end{pmatrix}
\begin{pmatrix}
    q & 1 \\
    q & 0 \\
    0 & 0 \\
    0 & 0
\end{pmatrix} =
\begin{pmatrix}
    q^6 + 2q^5 + 2q^4 + q^3 & q^5+q^4+q^3\\
    q^5+2q^4+q^3& q^4+q^3 \\
    q^4+q^3 & q^3
\end{pmatrix}
    \]
Compare the (1,1) entry here to Example \ref{ex:PP}, and note how it describes how the six vector partitions break up by number of parts.
\end{example}

\begin{remark}
    As before, each entry in the resulting matrix
    has an interpretation. Entry $(i,j)$ gives the
    $q$-analog for $\bfd'=(d_1+1-i,d_2,d_2,\ldots,d_{n-1},d_n+1-j)$. Multiplying by $q^{i+j-2}$ yields the $q$-analog for the subset
    of vector partitions in $\mathcal{P}_\bfd$ such that
    $\mathbf{e}_1$ (resp. $\mathbf{e}_m$) appear at least $i$ (resp. $j$) times.
\end{remark}

\section{Lattice of mixed dimer covers}\label{sec:md}

In this section, we investigate the partial order on mixed dimer covers defined in Section~\ref{sec:prelim}. Specifically, in
\begin{itemize}
    \item[$(1)$] Proposition~\ref{prop:poset} we verify that $(\mathcal{D}_\mathbf{d},\preceq_T)$ defines a poset structure;
    \item[$(2)$] Proposition~\ref{prop:uniquemin} we identify the unique minimal element of $(\mathcal{D}_\mathbf{d},\preceq_T)$; and
    \item[$(3)$] Theorem~\ref{thm:dlattice} we show that $(\mathcal{D}_\mathbf{d},\preceq_T)$ is a distributive lattice by constructing a poset $(P_\mathbf{d},\preceq_L)$ for which $(\mathcal{J}(P_\mathbf{d}),\subseteq)$ is isomorphic to $(\mathcal{D}_\mathbf{d},\preceq_T)$.
\end{itemize}
This section builds on the original work by Propp \cite{propp2002lattice} who demonstrated a lattice structure on single dimers (i.e. all $d_i = 1$),
work by Musiker, Ovenhouse, Schiffler, and Zhang \cite{zbMATH08171225} who discussed this structure for higher dimers (i.e. all $d_i = d$ for a fixed integer $d$), and work by Claussen and Ovenhouse \cite{ClaussenOvenhouse} concerning certain mixed dimer covers.

To start, we verify that the relation $\preceq_T$ of Definition~\ref{def:mdorder} does in fact define a partial order on $\mathcal{D}_{\mathbf{d}}$. For the proof, we require the following lemma, which shows that $\mathcal{D}_\mathbf{d}$ is closed under the appropriate twists.

\begin{lemma}
Let $D \in \mathcal{D}_{\mathbf{d}}$ with $\mathbf{d}\in\mathbb{Z}^n_{\ge 0}$. If $D'$ is the horizontal twist or vertical twist of $D$, then $D' \in \mathcal{D}_{\mathbf{d}}$.
\end{lemma}
\begin{proof}
 From Lemma \ref{lem:HoriztonalDeterminesMatching}, it is clear that if $D'$ is a horizontal twist of $D$, then $D' \in \mathcal{D}_{\mathbf{d}}$.

Now, suppose that $D'$ is the vertical twist of $D$ at tile $j\in [n-1]$. This means $v_j(D) > 0$ and $v_{j+1}(D) > 0$. Lemma \ref{lem:VsHsDs} implies that the following inequalities are strict: $h_{j-1}(D) + h_j(D) < d_j$ and $h_j(D) + h_{j+1}(D) < d_{j+1}$. Therefore, by Lemma \ref{lem:HoriztonalDeterminesMatching}, $D'$ is again an element of $\mathcal{D}_{\mathbf{d}}$.
 \end{proof}

\begin{proposition}\label{prop:poset}
The pair $(\mathcal{D}_{\mathbf{d}},\leqt)$ with $\mathbf{d}\in\mathbb{Z}^n_{\ge 0}$ forms a poset.
\end{proposition}

\begin{proof}
    The relation $\preceq$ is reflexive and transitive by definition. As for antisymmetry, note that a vertical twist at tile $i$ of a mixed dimer cover $D$ for $i$ even increases $h_i(D)$ by one, while a horizontal twist at tile $i$ for $i$ odd decreases $h_i(D)$ by one. Consequently, it follows that $\preceq$ is antisymmetric.
\end{proof}

 Next, we identify the unique minimal element of such posets.

\begin{definition}\label{def:Dmin}
For $\mathbf{d}\in\mathbb{Z}_{\ge 0}^n$, 
recall that $m_i=\min(d_i,d_{i+1})$.
Let $D_{\min} \in \mathcal{D}_{\mathbf{d}}$ be the mixed dimer cover satisfying \[
h_i(D_{\min}) = \begin{cases} m_i & i \text{ is odd}\\
0 & i \text{ is even}
\end{cases} \qquad \text{for}~i\in [n-1].
\]
\end{definition}

\begin{example}\label{ex:Dmin}
If $\mathbf{d} = (3,2,6,4,1,3)$, then $D_{\min}$ is as follows.

\begin{center}
\begin{tikzpicture}
\draw(0,0) to node[below]{2} (1,0) to  node[below]{0} (2,0)to  node[below]{4} (3,0) to  node[below]{0} (4,0) to node[below]{1}(5,0) to node[right]{2}(5,1) to node[above]{1} (4,1) to node[above]{0} (3,1) to node[above]{4} (2,1) to node[above]{0} (1,1) to node[above]{2} (0,1) to node[left]{1}(0,0);
\draw (1,0) to node[right]{0} (1,1);
\draw (2,0) to node[right]{2} (2,1);
\draw (3,0) to node[right]{0} (3,1);
\draw (4,0) to node[right]{0} (4,1);
\end{tikzpicture}
\end{center}
\end{example}

\begin{remark}\label{rmk:MultiplicityOfVertInDMin}
Notice that $v_i(D_{\min}) = d_i - m_i$ if $i$ is odd and $v_i(D_{\min}) = d_i - m_{i-1}$ if $i$ is even. 
\end{remark}

 \begin{proposition}\label{prop:uniquemin}
Let $D \in \mathcal{D}_{\mathbf{d}}$ with $\mathbf{d}\in\mathbb{Z}^n_{\ge 0}$. There exists a sequence of twists from $D_{\min}$ to $D$ consisting solely of horizontal twists at odd tiles and vertical twists at even tiles. Consequently, $D_{min}$ is the unique minimal element of $(\mathcal{D}_{\mathbf{d}},\preceq_T)$.
 \end{proposition}

\begin{proof}
Define $a_i(D)$ by \begin{equation}\label{eqn:min}
a_i(D) = \begin{cases}
m_i - h_i(D) & i \text{ odd}\\
h_i(D) & i \text{ even.}
\end{cases}
\end{equation}
We proceed via induction on $S(D) := \sum_{i=1}^{n-1}a_i(D)$. 

If $a_i(D) = 0$, then $h_i(D) = h_i(D_{\min})$ for all $1 \leq i \leq n-1$, and   by Lemma \ref{lem:HoriztonalDeterminesMatching}, $S(D) = 0$ if and only if $D = D_{\min}$. Here, the statement is clear.

Now, assume $S(D)>0$, implying $D \neq D_{\min}$. Suppose first that there exists an even number $i$ such that $a_i(D) > 0$. By Equation (\ref{eqn:min}), this implies that $h_i(D)$ is positive. Let $D'$ be the result of performing a horizontal twist on $D$ at tile $i$. We have $S(D') < S(D)$. By the inductive hypothesis, there is a sequence of twists from $D_{\min}$ to $D'$ using horizontal twists at odd tiles and vertical twists at even tiles. Appending a final vertical twist at tile $i$ to this sequence yields the desired sequence from $D_{\min}$ to $D$. 

Now, suppose that $S(D)>0$ and the only values $i$ such that $a_i(D) > 0$ are odd.  Let $j$ be the smallest such index. By our assumption, $a_{j-1}(D) = a_{j+1}(D) = 0$. Then, by Equation (\ref{eqn:min}), we have $h_{j-1}(D) = h_{j+1}(D) = 0$ and $h_j(D) < m_j = \min(d_j,d_{j+1})$. Therefore, \[
h_{j-1}(D) + h_j(D) < d_j \text{ and } h_j(D) + h_{j+1}(D) < d_{j+1},
\]
and from Lemma \ref{lem:VsHsDs}, we know that $v_j(D)$ and $v_{j+1}(D)$ are both positive. Now, let $D'$ be the result of performing a vertical twist at tile $j$. We again see that $S(D') < S(D)$, and, applying similar reasoning to that used previously, we can recover a sequence of twists of the desired form from $D_{\min}$ to $D$. 
\end{proof}

\begin{example}
Let $\mathbf{d} = (3,2,6,4,1,3)$, and let $D\in \mathcal{D}_{\mathbf{d}}$ be the following mixed dimer cover.
\begin{center}
\begin{tikzpicture}
\draw(0,0) to node[below]{1} (1,0) to  node[below]{1} (2,0)to  node[below]{1} (3,0) to  node[below]{1} (4,0) to node[below]{0}(5,0) to node[right]{3}(5,1) to node[above]{0} (4,1) to node[above]{1} (3,1) to node[above]{1} (2,1) to node[above]{1} (1,1) to node[above]{1} (0,1) to node[left]{2}(0,0);
\draw (1,0) to node[right]{0} (1,1);
\draw (2,0) to node[right]{4} (2,1);
\draw (3,0) to node[right]{2} (3,1);
\draw (4,0) to node[right]{0} (4,1);
\end{tikzpicture}
\end{center}
Recall, for this vector $\mathbf{d} = (3,2,6,4,1,3)$, $D_{\min}$ is provided in Example \ref{ex:Dmin}. Our sequence of values $a_i(D)$ is $(1,1,3,1,1)$. The proof of \Cref{prop:uniquemin} can be used to define a sequence of twists from $D$ to $D_{\min}$ using horizontal twists at even tiles and vertical twists at odd tiles. First, we maximally twist horizontally on even tiles. Here, this means performing one horizontal twist at tiles 2 and 4. The resulting dimer, $D'$, is below. 

\begin{center}
\begin{tikzpicture}
\draw(0,0) to node[below]{1} (1,0) to  node[below]{0} (2,0)to  node[below]{1} (3,0) to  node[below]{0} (4,0) to node[below]{0}(5,0) to node[right]{3}(5,1) to node[above]{0} (4,1) to node[above]{0} (3,1) to node[above]{1} (2,1) to node[above]{0} (1,1) to node[above]{1} (0,1) to node[left]{2}(0,0);
\draw (1,0) to node[right]{1} (1,1);
\draw (2,0) to node[right]{5} (2,1);
\draw (3,0) to node[right]{3} (3,1);
\draw (4,0) to node[right]{1} (4,1);
\end{tikzpicture}
\end{center}

Now, $D'$ agrees with $D_{\min}$ on all horizontal edges along even tiles, so that $a_i(D') =0 $ for all even $i$. Maximally twisting $D'$ vertically on odd tiles produces $D_{\min}$. Explicitly, we twist tile 1 one time, tile 3 three times, and tile 5 one time. 
\end{example}

\begin{remark}
 In light of Proposition \ref{prop:uniquemin} and the proof method, we see that there is also a unique maximal element in $(\mathcal{D}_{\mathbf{d}},\preceq_T)$. This is the element $D$ of $\mathcal{D}_{\mathbf{d}}$ with $h_i(D) = 0$ if $i$ is odd and $h_i(D) = m_i$ if $i$ is even. We can recognize this as the result of reversing the roles of even and odd in Definition \ref{def:Dmin}.
\end{remark}

\subsection{Distributive Lattices}

Given a vector $\mathbf{d} \in \mathbb{Z}_{\geq 0}^n$, we build a poset $(P_{\mathbf{d}},\preceq_L)$ as follows. The underlying set $P_{\mathbf{d}}$ is \[
\{(i,j)~:~1 \leq i \leq n-1,~1 \leq j \leq m_i\}
\]
and the partial order is given by taking the transitive closure of the relations
\begin{align*}
    (i,j)&\preceq_L(i,j+1) \text{ for all }1 \leq i \leq n-1,~1 \leq j \leq m_i-1,\\
    (i-1,j)&\preceq_L(i,j+d_i-m_{i-1}) \text{ for all even }i \text{ and all } 1 \leq j \leq m_i\text{, and}\\
    (i+1,j)&\preceq_L(i,j+d_{i+1}-m_{i+1}) \text{ for all even }i \text{ and all } 1 \leq j \leq m_i
\end{align*}
where in each case, there is no relation if one of the pairs is not an element of $P_{\mathbf{d}}$, i.e., an element of the form $(i,j)$ with $j > m_i$.

\begin{example}
For $\mathbf{d} = (3,2,6,4,1,3)$, the poset $(P_{\mathbf{d}},\preceq_L)$ is illustrated in Figure~\ref{fig:Pd}. 

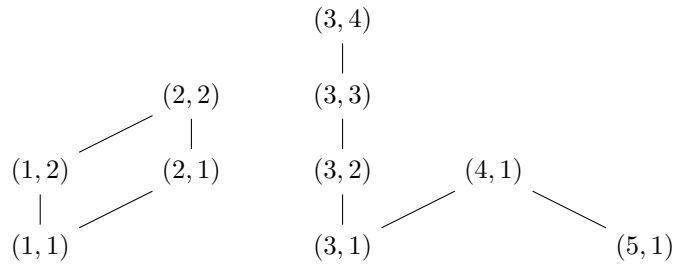
\begin{figure}[h]
    \centering
    \begin{tikzpicture}
\node(11) at (0,0){$(1,1)$};
\node(12) at (0,1){$(1,2)$};
\node(21) at (2,1){$(2,1)$};
\node(22) at (2,2){$(2,2)$};
\node(31) at (4,0){$(3,1)$};
\node(32) at (4,1){$(3,2)$};
\node(33) at (4,2){$(3,3)$};
\node(34) at (4,3){$(3,4)$};
\node(41) at (6,1){$(4,1)$};
\node(51) at (8,0){$(5,1)$};
\draw(11) -- (12);
\draw(21) -- (22);
\draw(31) -- (32);
\draw(32) -- (33);
\draw(33) -- (34);
\draw(11) -- (21);
\draw(12) -- (22);
\draw(41) -- (31);
\draw(41) -- (51);
\end{tikzpicture}
    \caption{$P_{\mathbf{d}}$ for $\mathbf{d} = (3,2,6,4,1,3)$}
    \label{fig:Pd}
\end{figure}
\end{example}

The posets $(P_{\mathbf{d}},\preceq_L)$ generalize a cross product of a zig-zag poset and a chain\footnote{The distributive lattices arising from the cross product of a zig-zag poset and a chain---in particular the sizes thereof---were studied in \cite{berman76}.}, where here the cover relations between adjacent copies of the chain have an offset.  We will refer to these as \emph{columns}; that is, the columns are the subposets of $(P_{\mathbf{d}},\preceq_L)$ on the collections $\{(i,j)~|~1\le j\le m_i\}$ for a fixed $1\le i\le n-1$. From Remark \ref{rmk:MultiplicityOfVertInDMin}, we see that the offset can also be read from the multiplicity of vertical edges in $D_{\min}$.
The posets $(P_{\mathbf{d}},\preceq_L)$ also resemble the $\Gamma$ posets in \cite{dilks2019rowmotion}.

\begin{lemma}\label{lem:CanDefineMapPsi}
Let $D \in \mathcal{D}_{\mathbf{d}}$ with $\mathbf{d}\in\mathbb{Z}_{\ge 0}^n$. Let $\Psi(D)$ be the subset of $P_{\mathbf{d}}$ of the form \[
\Psi(D):=\bigg(\bigcup_{i \text{ odd}}\{(i,1), (i,2) \ldots, (i,m_i - h_i(D)\}\bigg) \cup \bigg( \bigcup_{i \text{ even}}\{(i,1),(i,2),\ldots,(i,h_i(D))\}\bigg).
\]
The set $\Psi(D)$ forms an order ideal in $(P_{\mathbf{d}},\preceq_L)$.
\end{lemma}

\begin{proof}
Based on the nature of the relations generating the partial order on $P_{\mathbf{d}}$, it suffices to check that, for
$(i,j)\in\Psi(D)$ with $i$ even, the elements it covers
in adjacent odd-indexed columns, namely $(i-1,j-d_{i}+m_{i-1})$
and $(i+1,j-d_{i+1}+m_{i+1})$, are also in $\Psi(D)$.
More specifically, we can check for the elements covered by
the top element $(i,h_i(D))$ in the even-indexed column:
$(i-1,h_i(D) - d_{i} + m_{i-1})$ and $(i+1,h_i(D) - d_{i+1} + m_{i+1})$. 

By definition of the map $\Psi(D)$, the largest element in column $i+1$ is $(i+1,m_{i+1}-h_{i+1}(D))$. So checking whether $(i+1,h_i(D) - d_{i+1} + m_{i+1}) \in \Psi(D)$ is equivalent to checking \[
h_i(D) - d_{i+1} + m_{i+1} \leq m_{i+1} - h_{i+1}(D).
\]
This inequality is equivalent to $h_i(D) + h_{i+1}(D) \leq d_{i+1}$, which is guaranteed by Lemma \ref{lem:HoriztonalDeterminesMatching}.

The case for $(i-1,h_i(D) - d_{i} + m_{i-1})$ is similar.
 \end{proof}

 In light of Lemma \ref{lem:CanDefineMapPsi}, we can realize $\Psi$ as a map from $\mathcal{D}_{\mathbf{d}}$ to $\mathcal{J}(P_{\mathbf{d}})$. We claim that this in fact defines a poset isomorphism.

 \begin{theorem}\label{thm:dlattice}
The map $\Psi$ is an isomorphism of posets. In particular, $(\mathcal{D}_{\mathbf{d}},\leqt)$ with $\mathbf{d}\in\mathbb{Z}_{\ge 0}^n$ forms a distributive lattice.
 \end{theorem}

 \begin{proof}
Notice the second statement will follow from the first via Theorem~\ref{thm:FTFDL}.

Since the quantities $h_i(D)$ determine an element $D \in \mathcal{D}_{\mathbf{d}}$, the same can be said about the quantities $m_i-h_i(D)$ for $i$ odd and $h_i(D)$ for $i$ even. In particular, $\Psi$ is injective. Next, consider an order ideal $I$ of $(P_{\mathbf{d}},\preceq_L)$ such that the largest element in column $i$ is $(i,\ell_i)$. Note that for $i$ even, $$(i-1,\ell_i-d_i+m_{i-1})\preceq_L(i,(\ell_i-d_i+m_{i-1})+d_i-m_{i-1})=(i,\ell_i)$$ and $$(i+1,\ell_i-d_{i+1}+m_{i+1})\preceq_L(i,(\ell_i-d_{i+1}+m_{i+1})+d_{i+1}-m_{i+1})=(i,\ell_i).$$ Thus, since $I$ is an order ideal, for $i$ even,
\begin{equation}
\ell_{i-1}\geq \ell_{i} - d_i + m_{i-1}\qquad\text{and}\qquad\ell_{i+1} \geq \ell_{i} - d_{i+1} + m_{i+1}.\label{eq:OrderIdealinPd}
\end{equation}
Now, define a sequence $(h_1,\ldots,h_{n-1}) \in \mathbb{Z}_{\ge 0}^{n-1}$ such that $h_i = \ell_i$ for all even $i$, and $h_i = m_i - \ell_i$ for all odd $i$. If we rewrite the first part of Equation (\ref{eq:OrderIdealinPd}) in terms of $h_i$, we have, for even $i$, \[
m_{i-1} - h_{i-1} \geq h_i - d_i + m_{i-1}
\]
which is equivalent to $d_i \geq h_{i-1} + h_i$. A similar calculation can be carried out for $i+1$, yielding $d_i \geq h_{i+1} + h_i$. Finally, we naturally have $h_1 \leq m_1$ and $h_{n-1} \leq m_{n-1}$. Applying Lemma \ref{lem:HoriztonalDeterminesMatching}, there exists $D \in \mathcal{D}_{\mathbf{d}}$ such that $h_i = h_i(D)$ for all $i$, and, by construction, $\Psi(D) = I$.

We have at this point shown that $\Psi$ is bijective. It remains to check that it and its inverse preserve order. A cover relation in $(\mathcal{D}_{\mathbf{d}},\preceq_T)$ consists of a horizontal twist at an odd tile or a vertical twist at an even tile. If $D'$ is the result of a horizontal twist of $D$ at tile $i$, where $i$ is odd, then $h_i(D') = h_i(D) - 1$ and $h_j(D') = h_j(D)$ for all $j \neq i$. From the definition of $\Psi$ and since $i$ is odd, we see $\Psi(D') \supset \Psi(D)$. A similar argument shows the same if $D'$ is the result of a vertical twist at an even tile. In particular, $\Psi$ respects the cover relations, and therefore respects all relations. 

{Now, we check that $\Psi^{-1}$ is order-preserving. Take $I',I\in\mathcal{J}(P_\mathbf{d})$ where $I'$ covers $I$. Then $I\subset I'$ and $|I'\backslash I|=1$. Again let $\ell_i$ be such that  the largest element of $I$ in column $i$ is $(i,\ell_i)$. If no such element exists, let $\ell_i =0$. There exists $j$ such that, for all $i \neq j$, the largest element in column $i$ of $I'$ is $(i,\ell_i)$ whereas the largest element of $I'$ in column $j$ is $(j,\ell_j+1)$. There are two cases to consider depending on whether $j$ is even or odd.
\bigskip

\noindent
\textbf{Case 1:} $j$ even. In this case, if $\Psi^{-1}(I)=D$ and $\Psi^{-1}(I')=D'$, then $h_i(D)=h_i(D')$ for $i\neq j$ and $h_j(D)+1=h_j(D')$. Consequently, we have \[
h_{j-1}(D)+h_j(D)+1\le d_j=h_{j-1}(D)+h_j(D)+v_j(D)
\]
and
\[
h_{j+1}(D)+h_j(D)+1\le d_{j+1}=h_{j+1}(D)+h_j(D)+v_{j+1}(D)
\] 
where we set $h_{j+1}(D) = 0$ if $j = n-1$. This computation implies that  $v_j(D),v_{j+1}(D)>0$ which moreover means one can perform a vertical twist at tile $j$ of $D$. This twist yields $D'$, and by definition $D'$ covers $D$.
\bigskip

\noindent
\textbf{Case 2:} $j$ odd. In this case, if $\Psi^{-1}(I)=D$ and $\Psi^{-1}(I')=D'$, then $h_i(D)=h_i(D')$ for $i\neq j$ and $h_j(D)-1=h_j(D')$. Thus, $h_j(D)>0$ and one can perform a horizontal twist at tile $j$ of $D$ which results in $D'$ so that $D'$ covers $D$.} 
 \end{proof}

\section{A (partially) order-preserving bijection}\label{sec:opbij}

In this section, we provide a second bijection $\Phi$ between mixed dimer covers and rank tuples that better preserves the poset structure of each.  Let $\Phi(D)$ be the tuple $\Phi(D)=(r_1(D),\ldots,r_{n-1}(D))$ where \begin{align}
    r_i(D)=\begin{cases}
        h_i(D) & \text{for}~i~\text{even, and}\\
        \min(v_i(D),v_{i+1}(D)) & \text{for}~i~\text{odd.}
    \end{cases}.\label{eq:def_of_Phi}
\end{align}
Note that the tuple $\Phi(D)$ indeed satisfies the necessary inequalities for membership in $\mathcal{R}_\bfd$: $r_{i}+r_{i+1}\leq d_{i+1}$ for all $1\leq i\leq n-2$ and $r_i\leq m_i = \min(d_{i},d_{i+1})$ for all $2\leq i\leq n-1$.  

\begin{lemma}\label{lem:moving_up}
    Let $\mathbf{d}=(d_1,\hdots,d_n)$ and $\mathbf{r}=(r_1,r_2,\ldots,r_{n-1})\in \mathcal{R}_{\mathbf{d}}$. If \[\mathbf{r}'=(r_1,r_2,\ldots,r_{i-1},r_i+1,r_{i+1},\ldots,r_{n-1})\in \mathcal{R}_{\mathbf{d}}\] and $\mathbf{r}=\Phi(D)$ for some $D\in \mathcal{D}_{\mathbf{d}}$, then there exists $D'\in\mathcal{D}_{\mathbf{d}}$ such that $\Phi(D')=\mathbf{r}'$ and $D'\geqt D$.
\end{lemma}

\begin{proof}
We are concerned with the piece of the $\bfd$-dimer cover $D$ consisting of tiles $i-1$, $i$, and $i+1$
from the left, which we illustrate in the following diagram.

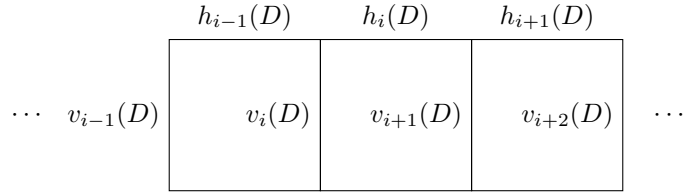
\[\begin{tikzpicture}[scale=1]
  \draw (0,0) rectangle (2,2);
  \draw (2,0) rectangle (4,2);
  \draw (4,0) rectangle (6,2);

  \node at (1,2) [above] {$h_{i-1}(D)$};
  \node at (3,2) [above] {$h_i(D)$};
  \node at (5,2) [above] {$h_{i+1}(D)$};
  \node at (0,1) [left]  {$v_{i-1}(D)$};
  \node at (2,1) [left] {$v_i(D)$};
  \node at (4,1) [left] {$v_{i+1}(D)$};
  \node at (6,1) [left] {$v_{i+2}(D)$};
  
  \node at (-1.5,1) [left] {$\cdots$};
  \node at (7,1) [left] {$\cdots$};
\end{tikzpicture}\]

Since $\mathbf{r}' \in \mathcal{R}_{\mathbf{d}}$, we must have strict inequalities $r_i < m_i$, $r_{i-1} + r_i < d_i$ and $r_i + r_{i+1} < d_{i+1}$.

In order to analyze further, we split into two cases based on the parity of $i$. Throughout the proof, set $h_0(D) = h_{n-1}(D) = 0$ and $v_0(D) = v_{n}(D) =0$.
First assume that $i$ is odd. Then $r_{i-1} = h_{i-1}(D)$ and $r_{i+1} = h_{i+1}(D)$ for $D$ such that $\Phi(D) = \mathbf{r}$. In this case, either
\begin{enumerate}
    \item[(i)] $\min(v_i(D),v_{i+1}(D))=v_i(D)$, in which case we consider
    \begin{align*}
        r_{i-1}'+r_i'=r_{i-1}+(r_i+1)&\le d_i,\\
        \intertext{and by definition of $\Phi$ and Lemma \ref{lem:VsHsDs}, we replace the left and right sides respectively to obtain}
        h_{i-1}(D)+v_i(D)+1 &\le h_{i-1}(D)+v_i(D)+h_i(D),
    \end{align*}
    so that $h_i(D) > 0$, or
    \item[(ii)] $\min(v_i(D),v_{i+1}(D))=v_{i+1}(D)$, in which case we consider
    \[r_i'+r_{i+1}'\leq d_{i+1}\]
    and similarly find that $h_i(D)>0$.
\end{enumerate}

Thus, in either case $h_i(D)$ is positive and one can apply a horizontal twist at tile $i$ of $D$ resulting in the desired $\mathbf{d}$-dimer cover $D'$.

Next assume that $i$ is even. In this case, we need to form two intermediate $\bfd$-dimer covers, $D_1$ and $D_2$ on our way to $D'$. The reader may wish to follow along using \Cref{ex:moving_up}. We form $D_1$ as follows. If $v_{i-1}(D)<v_i(D)$, then we simply let $D_1=D$ and note that $v_i(D_1)=v_i(D)>0$. If instead $v_{i-1}(D)\geq v_i(D)$, then we note that \begin{align*}
    h_{i-1}(D)&=d_i-(h_i(D)+v_i(D)) & \text{by Lemma \ref{lem:VsHsDs}}\\
    &=d_i-(r_{i}+r_{i-1}) & \text{by \Cref{eq:def_of_Phi}}\\
    &\geq1 & \text{by $\bfr'\in R_\bfd$.}
\end{align*}
This allows us to apply a horizontal twist at tile $i-1$ of $D$, yielding $D_1\in\mathcal{D}_\bfd$. We make the following observations about $D_1$ that hold in both cases: \begin{enumerate}
    \item[(a1)] $D_1\geqt D$,
    \item[(b1)] $v_i(D_1)>0$, and
    \item[(c1)] $v_i(D_1)-v_i(D)=\delta_{v_{i-1}(D)\geq v_i(D)}$ and similarly for $v_{i-1}$.
\end{enumerate}
where for a statement $S$, we have that $\delta_S$
is the quantity $1$ when $S$ is true and the quantity
$0$ when $S$ is false.

Now, we form $D_2$ from $D_1$ similarly. If $v_{i+2}(D_1)<v_{i+1}(D_1)$, then we simply let $D_2=D_1$ and note that $v_{i+1}(D_2)=v_{i+1}(D_1)>0$. If instead $v_{i+2}(D_1)\geq v_{i+1}(D_1)$, then we note that
\begin{align*}
    h_{i+1}(D_1)=h_{i+1}(D)&=d_i-(h_i(D)+v_{i+1}(D)) & \text{by Lemma \ref{lem:VsHsDs}}\\
    &=d_i-(r_i+r_{i+1})& \text{by \Cref{eq:def_of_Phi}}\\
    &\geq 1 & \text{by $\bfr'\in R_\bfd$.}
\end{align*}
This allows us to apply a horizontal twist at tile $i+1$ of $D_1$, yielding $D_2\in\mathcal{D}_\bfd$. We make the
following observations about $D_2$ that hold in both cases:\begin{enumerate}
    \item[(a2)] $D_2\geqt D_1$,
    \item[(b2)] $v_{i+1}(D_2)>0$, and
\end{enumerate}

Finally, because $v_i(D_2)=v_i(D_1)>0$ and $v_{i+1}(D_2)>0$, we form a $\bfd$-dimer cover $D'$ from $D_2$ by performing a vertical twist at tile $i$. Note that by transitivity, $D'\succ_T D$.
Because we changed values only at squares $i-1$, $i$, and $i+1$, we have that $\Phi(D')$ can only differ from $\bfr=\Phi(D)$ at indices $i-1$, $i$ and $i+1$. Our final twist at tile $i$ ensures that 
\begin{align*}
    \Phi(D')_i&=h_i(D)+1=r_i+1.
\end{align*}
Now, working back through the modifications of each twist, we have that
\begin{align*}
    \Phi(D')_{i-1}&=\min(v_{i-1}(D'),v_{i}(D'))\\
    &=\min(v_{i-1}(D_2),v_i(D_2)-1)\\
    &=\min(v_{i-1}(D_1),v_i(D_1)-1)\\
    &=\min(v_{i-1}(D),v_i(D)-1)+\delta_{v_{i-1}(D)\geq v_i(D)}.
\end{align*}
If $v_{i-1}(D)\geq v_i(D)$, this resolves to $v_{i}(D)-1+1=v_{i}(D)$.
If instead $v_{i-1}(D)< v_i(D)$, this resolves to $v_{i-1}(D)$. Hence, $\Phi(D')_{i-1}=\min(v_{i-1}(D),v_i(D))=r_{i-1}$. Similarly, $\Phi(D')_{i+1}=r_{i+1}$ and $\Phi(D')=\bfr'$ as desired.

We have demonstrated the result for all values of $i$ other than $i=1$ and $i=n$. These two remaining cases are analogous and in fact simpler. \qedhere

\end{proof}

\begin{example}\label{ex:moving_up}
Below we show the process described in \Cref{lem:moving_up} for a particular $\bfd$-dimer cover
with $\bfd=(4,5,7,3)$. At each stage, the most recently twisted cell is
shaded. Note that because $v_4(D_1)<v_3(D_1)$, we have no twist in the subsequent step,
so $D_2=D_1$.

\[\begin{array}{ccc}
    \raisebox{.34in}{$D=$}\begin{tikzpicture}[scale=.5]
      \draw (0,0) rectangle (2,2);
      \draw (2,0) rectangle (4,2);
      \draw (4,0) rectangle (6,2);
      \node at (1,2) [above] {$2$};
      \node at (3,2) [above] {$3$};
      \node at (5,2) [above] {$2$};
      \node at (1,0) [below] {$2$};
      \node at (3,0) [below] {$3$};
      \node at (5,0) [below] {$2$};
      \node at (0,1) [left]  {$2$};
      \node at (2,1) [left] {$0$};
      \node at (4,1) [left] {$2$};
      \node at (6,1) [left] {$1$};
    \end{tikzpicture} &  \raisebox{.34in}{$\longleftrightarrow$} & \raisebox{.34in}{$\Phi(D)=(0,3,1)$}\\
        \raisebox{.34in}{$D_1=$}\begin{tikzpicture}[scale=.5]
      \draw[fill=red!40!white] (0,0) rectangle (2,2);
      \draw (2,0) rectangle (4,2);
      \draw (4,0) rectangle (6,2);
      \node at (1,2) [above] {$1$};
      \node at (3,2) [above] {$3$};
      \node at (5,2) [above] {$2$};
      \node at (1,0) [below] {$1$};
      \node at (3,0) [below] {$3$};
      \node at (5,0) [below] {$2$};
      \node at (0,1) [left]  {$3$};
      \node at (2,1) [left] {$1$};
      \node at (4,1) [left] {$2$};
      \node at (6,1) [left] {$1$};
    \end{tikzpicture} & & \\
        \raisebox{.34in}{$D_2=$}\begin{tikzpicture}[scale=.5]
      \draw (0,0) rectangle (2,2);
      \draw (2,0) rectangle (4,2);
      \draw (4,0) rectangle (6,2);
      \node at (1,2) [above] {$1$};
      \node at (3,2) [above] {$3$};
      \node at (5,2) [above] {$2$};
      \node at (1,0) [below] {$1$};
      \node at (3,0) [below] {$3$};
      \node at (5,0) [below] {$2$};
      \node at (0,1) [left]  {$3$};
      \node at (2,1) [left] {$1$};
      \node at (4,1) [left] {$2$};
      \node at (6,1) [left] {$1$};
    \end{tikzpicture} &  & \\
        \raisebox{.34in}{$D'=$}\begin{tikzpicture}[scale=.5]
      \draw (0,0) rectangle (2,2);
      \draw[fill=red!40!white] (2,0) rectangle (4,2);
      \draw (4,0) rectangle (6,2);
      \node at (1,2) [above] {$1$};
      \node at (3,2) [above] {$4$};
      \node at (5,2) [above] {$2$};
      \node at (1,0) [below] {$1$};
      \node at (3,0) [below] {$4$};
      \node at (5,0) [below] {$2$};
      \node at (0,1) [left]  {$3$};
      \node at (2,1) [left] {$0$};
      \node at (4,1) [left] {$1$};
      \node at (6,1) [left] {$1$};
    \end{tikzpicture} &  \raisebox{.34in}{$\longleftrightarrow$} & \raisebox{.34in}{$\Phi(D')=(0,4,1)$}\\
\end{array}\]
\end{example}

\begin{theorem}
\label{thm:Phiorderpreserving}
    The map $\Phi:\mathcal{D}_{\mathbf{d}}\to \mathcal{R}_{\mathbf{d}}$ is invertible and its inverse is order-preserving. 
\end{theorem}

\begin{proof}
    The poset $(\mathcal{D}_\bfd,\leqt)$ has a unique smallest element
    $D_{\min}$. Following the definition of $D_{\min}$ and $\Phi$, we have that $\Phi(D_{\min})=(0,0,\ldots,0)$, the smallest element of
    $(\mathcal{R}_\bfd,\leqd)$. Consequently, repeated application of \Cref{lem:moving_up} shows that the map $\Phi$ is surjective. Considering \Cref{cor:naive_bijection}, which shows the two sets have the same cardinality,
    it follows that $\Phi$ is a bijection.

    Furthermore, if $\bfr,\bfr'\in\mathcal{R}_\bfd$ and $\bfr\leqd \bfr'$, then there is a sequence of cover relations \[\bfr=\bfr^{(1)}\lessdotd\bfr^{(2)}\lessdotd\cdots\lessdotd \bfr^{(k)}=\bfr'.\]

    At each cover relation, we apply \Cref{lem:moving_up} to create a chain of
    inequalities
    \[D=D^{(1)}\leqt D^{(2)}\leqt\cdots\leqt D^{(k)}=D'\] where
    $\bfr=\Phi(D)$ and $\bfr'=\Phi(D')$. Hence $\Phi^{-1}(\bfr)\leqt\Phi^{-1}(\bfr')$,
    meaning the inverse $\Phi^{-1}$ is order-preserving.
\end{proof}

\begin{example}

 Let $\mathbf{d} = (1,1,1,1)$ and consider the sets $\mathcal{R}_{\mathbf{d}}$ and $\mathcal{D}_{\mathbf{d}}$ . Each has five elements and $\Phi^{-1}$ is one bijection between them. The Hasse diagram on the partial order of each is drawn below. We see that these partial orders are not isomorphic.  In particular, $(0,1,0)$ and $(1,0,1)$ are incomparable in $\mathcal{R}_{\mathbf{d}}$ but $\Phi^{-1}((1,0,1)) \prec_T \Phi^{-1}(0,1,0)$.

    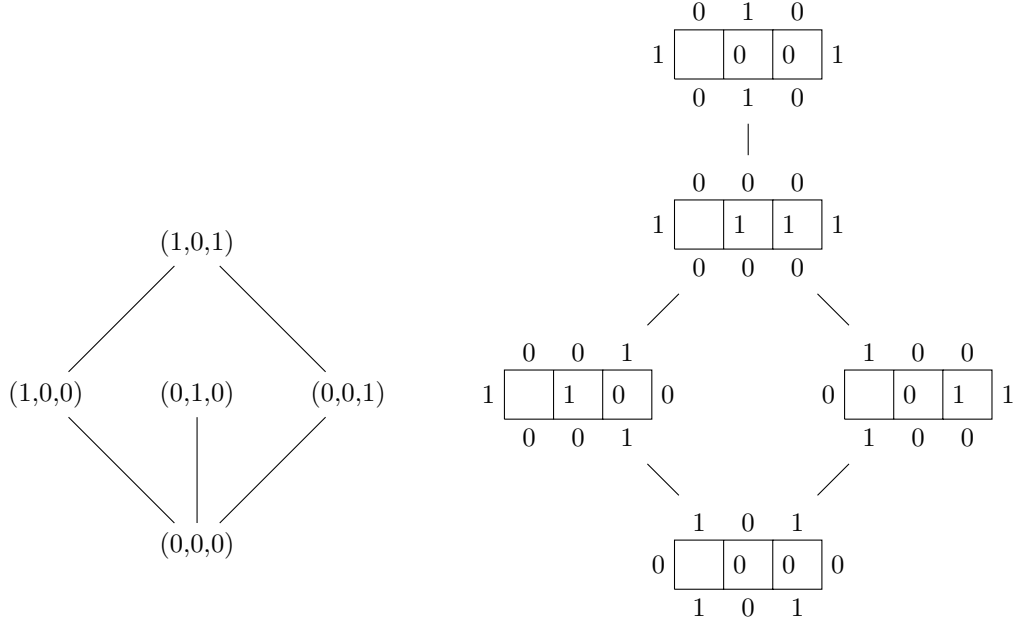
\begin{figure}[H]
    \centering
\begin{minipage}{.3\textwidth}
\vspace{21mm}

    $$\begin{tikzpicture}[scale=1]
        \node (1) at (0,0) {(0,0,0)};
        \node (2) at (-2,2) {(1,0,0)};
        \node (3) at (0,2) {(0,1,0)};
        \node (4) at (2,2) {(0,0,1)};
        \node (5) at (0,4) {(1,0,1)};
\draw (1)--(2)--(5)
(1)--(3)
(1)--(4)--(5)
;
    \end{tikzpicture}$$
\end{minipage}\begin{minipage}{.6\textwidth}
       $$\begin{tikzpicture}[scale=.75]
        \node (1) at (0,0) {\begin{tikzpicture}[scale=.65]
\draw(0,0) to node[below]{1} (1,0) to  node[below]{0} (2,0) to  node[right]{0} (2,1) to  node[above]{0} (1,1) to node[above]{1}(0,1) to node[left]{0}(0,0)

(2,0) to  node[below]{1} (3,0) to  node[right]{0} (3,1) to  node[above]{1} (2,1)
(1,0) to node[right]{0} (1,1);
\end{tikzpicture}};
        \node (2) at (-3,3) {\begin{tikzpicture}[scale=.65]
\draw(0,0) to node[below]{0} (1,0) to  node[below]{0} (2,0) to  node[right]{0} (2,1) to  node[above]{0} (1,1) to node[above]{0}(0,1) to node[left]{1}(0,0)

(2,0) to  node[below]{1} (3,0) to  node[right]{0} (3,1) to  node[above]{1} (2,1)
(1,0) to node[right]{1} (1,1);
\end{tikzpicture}};
        \node (3) at (3,3) {\begin{tikzpicture}[scale=.65]
\draw(0,0) to node[below]{1} (1,0) to  node[below]{0} (2,0) to  node[right]{1} (2,1) to  node[above]{0} (1,1) to node[above]{1}(0,1) to node[left]{0}(0,0)

(2,0) to  node[below]{0} (3,0) to  node[right]{1} (3,1) to  node[above]{0} (2,1)
(1,0) to node[right]{0} (1,1);
\end{tikzpicture}};
        \node (4) at (0,6) {\begin{tikzpicture}[scale=.65]
\draw(0,0) to node[below]{0} (1,0) to  node[below]{0} (2,0) to  node[right]{1} (2,1) to  node[above]{0} (1,1) to node[above]{0}(0,1) to node[left]{1}(0,0)

(2,0) to  node[below]{0} (3,0) to  node[right]{1} (3,1) to  node[above]{0} (2,1)
(1,0) to node[right]{1} (1,1);
\end{tikzpicture}};
        \node (5) at (0,9) {\begin{tikzpicture}[scale=.65]
\draw(0,0) to node[below]{0} (1,0) to  node[below]{1} (2,0) to  node[right]{0} (2,1) to  node[above]{1} (1,1) to node[above]{0}(0,1) to node[left]{1}(0,0)

(2,0) to  node[below]{0} (3,0) to  node[right]{1} (3,1) to  node[above]{0} (2,1)
(1,0) to node[right]{0} (1,1);
\end{tikzpicture}};
\draw (1)--(2)
(1)--(3)
(2)--(4)
(3)--(4)--(5)
;
    \end{tikzpicture}$$
    
\end{minipage}\caption{The bijection $\Phi^{-1}$ maps the diamond in the left poset to the diamond in the right poset. The ``extra'' element $(0,1,0)$ maps to the maximal dimer cover on the right. Although $\Phi^{-1}$ is order preserving, $\Phi$ itself is not.}
    \label{fig:inverseorderpreserving}
\end{figure}

\end{example}

\begin{remark}
    Notice that even though $\Phi$ is not an isomorphism of posets, it is not far
    off in the following sense. For $\bfr\lessdotd\bfr'$, we have from \Cref{lem:moving_up} that if $\bfr$ and $\bfr'$ differ at an odd index,
    then $\Phi(\bfr)\lessdott\Phi(\bfr')$. If instead they differ at an
    even index, then $\Phi(\bfr)$ and $\Phi(\bfr')$ are separated by at most three cover relations in the twist order.
\end{remark}

\section{The case where $\mathbf{d} = (d^n)$}

\label{sec:uniformvector}

The classical setting of a dimer cover, i.e. a \emph{perfect matchings}, in our notation is when $\mathbf{d} = (1^n)$. Here, the cardinality of $\mathcal{D}_{\mathbf{d}}$ is a Fibonacci number. Perfect matchings of snake graphs in general, of which the $2 \times n$ grid is a special case, have been an extremely useful tool in the theory of cluster algebras (see for example \cite{musiker2010cluster}) and Markov numbers (such as in \cite{LLRS}). 

Mixed dimer covers in $\mathcal{D}_{\mathbf{d}}$ when $\bfd = (d^n)$ (i.e. $d$-dimer covers) for $d \geq 1$ have been prominent in the literature lately. Musiker, Ovenhouse, and Zhang showed that double dimer covers can be used to study a super analogue of type $A$ cluster algebras \cite{zbMATH07567672}. Soon after, these same authors along with Schiffler studied the combinatorics of general $d$-dimer covers  \cite{zbMATH08171225}. The latter reference in particular examines the lattice structure in this case. 

\subsection{Glass Plates}
When restricting to the case of a uniform dimension vector
$\bfd=(d^n)$, a connection with a surprising set of
combinatorial objects appears: Imagine a beam of light
shone through $n$ glass plates stacked atop each other.
Upon meeting the surface of each plate, some of the light shines through
and some of it is reflected. A path is a choice,
for each time light from the beam meets a surface,
to either bounce or shine through. We model
such a scenario with $n+1$ lines representing
the (sometimes shared) surfaces of the glass plates.
See \cite{hoggatt79} for an introduction to the study of such objects and \cite[A006356]{OEIS} for the enumeration in the case of three glass plates.

\begin{example}\label{ex:light_beams}

Here we illustrate the six paths a beam of light
can take through three glass plates with exactly
two bounces. Under each image is the tuple recording
the number of times each segment crosses a boundary shared
by two plates.

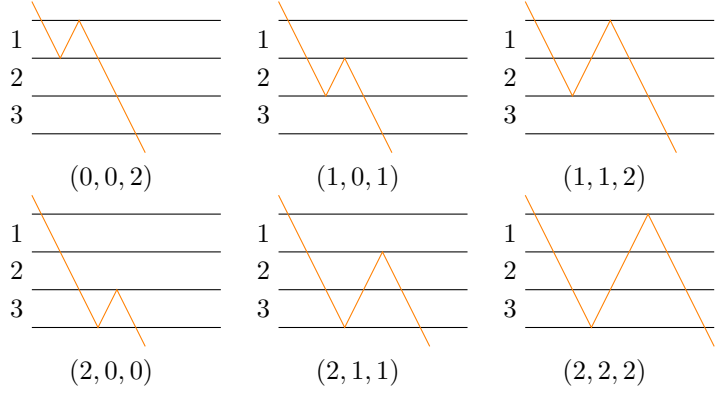
\[\begin{array}{ccc}
     \begin{tikzpicture}
        \draw[black] (0,0) -- (2.5,0);
        \draw[black] (0,.5) -- (2.5,.5);
        \draw[black] (0,1) -- (2.5,1);
        \draw[black] (0,1.5) -- (2.5,1.5);
        \node at (-.2,1.25) {1};
        \node at (-.2,.75) {2};
        \node at (-.2,.25) {3};
        \draw[orange] (0,1.75) -- (.375,1) -- (.625,1.5) -- (1.5,-0.25);
    \end{tikzpicture}  & \begin{tikzpicture}
        \draw[black] (0,0) -- (2.5,0);
        \draw[black] (0,.5) -- (2.5,.5);
        \draw[black] (0,1) -- (2.5,1);
        \draw[black] (0,1.5) -- (2.5,1.5);
        \node at (-.2,1.25) {1};
        \node at (-.2,.75) {2};
        \node at (-.2,.25) {3};
        \draw[orange] (0,1.75) -- (.625,.5) -- (.875,1) -- (1.5,-0.25);
    \end{tikzpicture}  & \begin{tikzpicture}
        \draw[black] (0,0) -- (2.5,0);
        \draw[black] (0,.5) -- (2.5,.5);
        \draw[black] (0,1) -- (2.5,1);
        \draw[black] (0,1.5) -- (2.5,1.5);
        \node at (-.2,1.25) {1};
        \node at (-.2,.75) {2};
        \node at (-.2,.25) {3};
        \draw[orange] (0,1.75) -- (.625,.5) -- (1.125,1.5) -- (2,-0.25);
    \end{tikzpicture} \\
    (0,0,2) & (1,0,1) & (1,1,2) \\
    \begin{tikzpicture}
        \draw[black] (0,0) -- (2.5,0);
        \draw[black] (0,.5) -- (2.5,.5);
        \draw[black] (0,1) -- (2.5,1);
        \draw[black] (0,1.5) -- (2.5,1.5);
        \node at (-.2,1.25) {1};
        \node at (-.2,.75) {2};
        \node at (-.2,.25) {3};
        \draw[orange] (0,1.75) -- (.875,0) -- (1.125,.5) -- (1.5,-0.25);
    \end{tikzpicture} & \begin{tikzpicture}
        \draw[black] (0,0) -- (2.5,0);
        \draw[black] (0,.5) -- (2.5,.5);
        \draw[black] (0,1) -- (2.5,1);
        \draw[black] (0,1.5) -- (2.5,1.5);
        \node at (-.2,1.25) {1};
        \node at (-.2,.75) {2};
        \node at (-.2,.25) {3};
        \draw[orange] (0,1.75) -- (.875,0) -- (1.375,1) -- (2,-0.25);
    \end{tikzpicture} & \begin{tikzpicture}
        \draw[black] (0,0) -- (2.5,0);
        \draw[black] (0,.5) -- (2.5,.5);
        \draw[black] (0,1) -- (2.5,1);
        \draw[black] (0,1.5) -- (2.5,1.5);
        \node at (-.2,1.25) {1};
        \node at (-.2,.75) {2};
        \node at (-.2,.25) {3};
        \draw[orange] (0,1.75) -- (.875,0) -- (1.625,1.5) -- (2.5,-0.25);
    \end{tikzpicture}\\
    (2,0,0) & (2,1,1) & (2,2,2)
\end{array}\]

\end{example}

\begin{proposition}\label{prop:glass_plates}
    The set $\mathcal{P}_\bfd$ for $\bfd=(d^n)$ is
    equinumerous with the set of paths taken by
    a beam of light shone through $d+1$ parallel glass
    plates in which the beam bounces $n-1$ times.
\end{proposition}

\begin{proof}
    A path with $n-1$ bounces is broken up by these bounces
    into $n$ line segments. Let $c_i$ for $1\leq i\leq n$
    be the number of shared surfaces crossed by the $i$th
    segment (in the illustrations in \Cref{ex:light_beams},
    this is the number of times the line segment crosses a
    horizontal line other than the very top or very bottom
    line). Note that a path is completely determined by
    the sequence $(c_1,c_2,\ldots,c_{n-1})\in[d]^{n-1}$.
    From this tuple, we define $\bfr=(r_1,r_2,\ldots,r_{n-1})$
    by
    \begin{align*}
        r_1&=d-c_1\text{, and}\\
        r_i&=d-c_i-r_{i-1}\text{ for }1<i\leq n-1.
    \end{align*}
    Note that for any sequence $(c_1,c_2,\ldots,c_{n-1})\in[d]^{n-1}$,
    we have that $r_i+r_{i+1}=d-c_{i+1}\leq d$ for all $1\leq i\leq n-2$. Thus, it remains to show that the sequences that determine
    a path of light are precisely those that satisfy the final
    condition to be a rank tuple: $1\leq r_k\leq d$ for
    all $1\leq k\leq n-1$.
    
    Note that we can unpack the above recursive definition for the $r_k$ to obtain
    \[r_k=\begin{cases}
        \sum_{i=1}^k (-1)^{i-1}c_i & \text{if $k$ is even, and}\\
        d-\sum_{i=1}^k (-1)^{i-1}c_i & \text{if $k$ is odd.}
    \end{cases}\]
    The condition that $0\leq r_k\leq d$ for all $k$ then becomes
    \[0\leq\sum_{i=1}^k (-1)^{i+1}c_i\leq d\] for all $1\leq k\leq n-1$, which
    is precisely the condition that the beam of light does
    not leave the glass plates until it has completed $n-1$
    bounces. The result follows.
\end{proof}

\subsection{Recurrence}

Let $a_d(n) := \vert  
\mathcal{D}_{(d^n)}
\vert$. As a convention, we set $a_d(0) = 1$. For small values of $d$, there are recorded recurrence relations for $a_d(n)$. If $d = 1$, we recover the Fibonnaci recurrence \[
a_1(n) = a_1(n-1) + a_1(n-2),
\]
and when $d = 2$, Musiker, Ovenhouse, Schiffler, and Zhang \cite{zbMATH08171225} used their enumeration formula to show \[
a_2(n) = 2a_2(n-1) + a_2(n-2) - a_2(n-3).
\]
Recurrences for $d=3,4,5$ are given in OEIS  A006357, A006358, and A006359, respectively. 
In the remainder of this section, we present a way to leverage Theorem~\ref{thm:ClaussenOvenhouse} to give the recurrence in general.

First, note that in the case that $\bfd=(d^n)$, Theorem~\ref{thm:ClaussenOvenhouse} tells us that
$a_d(n)$ is the $(1,1)$-entry of the matrix
\[\left(R_{d,d}\right)^n.\]
If $A_d(x)=\sum_{n\geq0}a_d(n)x^n$ is the ordinary generating function for the sequence
$(a_d(n))_{n\geq0}$, then $A_d(x)$ is the $(1,1)$-entry of the matrix
\[\sum_{n\geq0}\left(R_{d,d}\right)^nx^n.\]
By standard matrix identities, this matrix is equal to the inverse of a particular matrix:
\[\sum_{n\geq0}\left(R_{d,d}\right)^nx^n=(I_{d+1}-xR_{d,d})^{-1}\] where $I_{d+1}$ is the
$(d+1)\times(d+1)$ identity matrix.
Let $M_d(x)=I_{d+1}-xR_{d,d}$.
Because we need only the $(1,1)$-entry of its inverse, we apply Cramer's rule to obtain this entry as
the quotient
\[A_d(x)=\frac{\det\left(M_d(x)_{1,1}\right)}{\det\left(M_d(x)\right)}\]
where for a matrix $M$, we write $M_{i,j}$ for the matrix obtained by deleting
the $i$th row and $j$th column of $M$.

\begin{example}
    For $d=4$, we have
    \begin{align*}
        M_4(x)&=\left[\begin{array}{rrrrr}
-x + 1 & -x & -x & -x & -x \\
-x & -x + 1 & -x & -x & 0 \\
-x & -x & -x + 1 & 0 & 0 \\
-x & -x & 0 & 1 & 0 \\
-x & 0 & 0 & 0 & 1
\end{array}\right]\text{, and}\\
M_4(x)_{1,1}&=\left[\begin{array}{rrrr}
-x + 1 & -x & -x & 0 \\
-x & -x + 1 & 0 & 0 \\
-x & 0 & 1 & 0 \\
 0 & 0 & 0 & 1
\end{array}\right].
    \end{align*}
    Note that computing the determinant via cofactor expansion
    along the last row of $M_4(x)_{1,1}$ yields the determinant
    of the smaller matrix $M_2(x)$. So in this case,
    \[A_4(x)=\frac{\det(M_2(x))}{\det(M_4(x))}=\frac{1-2x-x^2+x^3}{1-3x-3x^2+4x^3+x^4-x^5}.\]
\end{example}

\begin{proposition}\label{prop:CramersRuleAd}
    If $M_d(x)=I_{d+1}-xR_{d,d}$, then
    \[A_d(x)=\frac{\det(M_{d-2}(x))}{\det(M_d(x))}.\]

        \label{prop:det}
\end{proposition}

\begin{proof}
Notice that the last row of $M_d(x)$ is only nonzero in the first and last column, which have entries $-x$ and 1 respectively. Therefore applying cofactor expansion along the last row of $M_d(x)_{1,1}$ shows that the determinant of $M_d(x)_{1,1}$ is the minor of $M_{d}(x)$ in rows and columns $2,3,\ldots,d$. By the symmetry in the definition of $M_d(x)$, we see that this minor is exactly the determinant of $M_{d-2}(x)$, as desired.
\end{proof}

\begin{theorem}\label{thm:gen_fun_formula}
    For $M_d(x)=I_{d+1}-xR_{d,d}$, we have that
    \[\det(M_d(x))=\det(M_{d-2}(x))-x\det(M_{d-1}(-x)).\]
    Furthermore,
    \[A_d(x)=\frac{1 - \sum_{i=1}^{d-1} (-1)^{\lfloor \frac{i-1}{2}\rfloor} {\lfloor \frac{d+i-1}{2}\rfloor \choose i}x^i}{1 - \sum_{i=1}^{d+1} (-1)^{\lfloor \frac{i-1}{2}\rfloor} {\lfloor \frac{d+i+1}{2}\rfloor \choose i}x^i}.\]
\end{theorem}

\begin{proof}
    Let $U_{d+1}$ be the $(d+1)\times(d+1)$ upper-triangular matrix consisting
    of ones on and above the diagonal
    and zeros below the diagonal. Let $J_{d+1}$ be the $(d+1)\times(d+1)$ matrix with ones on the
    antidiagonal and zeros elsewhere. Then we have that $R_{d,d}=U_{d+1}J_{d+1}$,
    hence
    \begin{align*}
        M_d(x)&=I_{d+1}-xU_{d+1}J_{d+1}\\
        &=U_{d+1}\left({U_{d+1}}^{-1}-xJ_{d+1}\right).
    \end{align*}
    Because $\det(U_{d+1})=1$, we then have that the determinant of $M_d(x)$
    is the same as the determinant of
    \[N_d(x):={U_{d+1}}^{-1}-xJ_{d+1}.\] Note that ${U_{d+1}}^{-1}$ consists of ones on the diagonal,
    negative ones on the superdiagonal, and zeros elsewhere. Hence,
    the first column of $N_d(x)$ consists of a one in the first row and
    a $-x$ in the last row.

For the computation of $\det(M_d(x)) = \det(N_d(x))$, the reader may wish to follow along using Example \ref{ex:determinantofN_5}.
    We proceed by cofactor expansion down the first column: 
    \begin{equation}
    \label{eq:determinant}
         \det(N_d(x)) = 1\cdot\det(N_d(x)_{1,1}) + (-1)^{d+1+1}\cdot (-x)\cdot \det(N_d(x)_{d+1,1}).
    \end{equation}
    Note that $N_d(x)_{1,1}$ is a $d\times d$ matrix whose last row consists of all zeros except for a 1 in the final position. Consequently, expanding along this row, we have 
    $$\det(N_d(x)_{1,1}) = (-1)^{d+d}\cdot 1 \cdot \det((N_d(x)_{1,1})_{d,d})$$
    and, similarly as in the proof of Proposition \ref{prop:det}, note that $(N_d(x)_{1,1})_{d,d} = N_{d-2}(x)$; therefore, we have $\det(N_d(x)_{1,1}) = \det(N_{d-2}(x))$.

    Now, for $N_d(x)_{d+1,1}$, note that this matrix is the result of transforming $N_{d-1}(x)$ by reversing the order of all rows and columns (corresponding to left and right multiplication by $J_{d}$), negating all entries, and then replacing $x$ with $-x$.
That is, \[N_d(x)_{d+1,1}=-J_{d}N_{d-1}(-x)J_{d}.\] Thus, \begin{align*}
    \det(N_d(x)_{d+1,1}) &=\det(-J_{d}N_{d-1}(-x)J_{d}) \\
    &=(-1)^{d}\det(J_{d})\det(N_{d-1}(-x))\det(J_{d})\\
    &=(-1)^{d}(-1)\det(N_{d-1}(-x))(-1) \\
    &=(-1)^{d+2}\det(N_{d-1}(-x)).
\end{align*}
Collecting these determinants and substituting back into Equation (\ref{eq:determinant}), we have
\begin{align*}
    \det(N_d(x)) &= \det(N_{d-2}(x)) + (-1)^{d+2}\cdot (-x)\cdot (-1)^{d+2}\det(N_{d-1}(-x))\\
    &=\det(N_{d-2}(x))  -x\det(N_{d-1}(-x)).
\end{align*}
Since $\det(M_d(x)) = \det(N_d(x))$, we obtain the recurrence
$$
\det(M_d(x)) = \det(M_{d-2}(x))  -x\det(M_{d-1}(-x))
$$
as desired.

    Next, we seek a closed formula for $A_d(x)$. Let 
    \[p_d(x) = 1 - \sum_{i=1}^{d+1} (-1)^{\lfloor \frac{i-1}{2}\rfloor} {\lfloor \frac{d+i+1}{2}\rfloor \choose i}x^i.\]
    By Proposition \ref{prop:CramersRuleAd}, it suffices to show $\det(M_d(x)) = p_d(x)$. Since we have provided a two-term recurrence for $\det(M_d(x))$, we now wish to  show $p_d(x)$ shares the same recurrence and initial conditions. The latter are easy to check: $p_{-1}(x) = 1$ and $p_0(x) = 1-x$.

    Now, we consider the right-hand side of the two-term recurrence but plug in $p_d(x)$: the expression
    $p_{d-2}(x)-xp_{d-1}(-x)$ expands as \begin{align*}
    &\phantom{=}\left(1 - \sum_{i=1}^{d-1} (-1)^{\lfloor \frac{i-1}{2}\rfloor} {\lfloor \frac{d+i-1}{2}\rfloor \choose i}x^i\right) - x \left(1-\sum_{i=1}^{d} (-1)^{\lfloor \frac{i-1}{2}\rfloor} {\lfloor \frac{d+i}{2}\rfloor \choose i}(-x)^i\right)\\
     &= 1-x-\left(\sum_{i=1}^{d-1} (-1)^{\lfloor \frac{i-1}{2}\rfloor} {\lfloor \frac{d+i-1}{2}\rfloor \choose i}x^i\right) - \left(\sum_{i=1}^{d} (-1)^{\lfloor \frac{i-1}{2}\rfloor} {\lfloor \frac{d+i}{2}\rfloor \choose i}(-x)^{i+1}\right).\\
     \intertext{Peeling off the $i=1$ term of the first sum and re-indexing
     the second sum then yields}
     &= 1-\left(\left\lfloor \frac{d}{2}\right\rfloor+1\right)x-\left(\sum_{i=2}^{d-1} (-1)^{\lfloor \frac{i-1}{2}\rfloor} {\lfloor \frac{d+i-1}{2}\rfloor \choose i}x^i\right)
     - \left(\sum_{i=2}^{d+1} (-1)^{\lfloor \frac{i-2}{2}\rfloor} {\lfloor \frac{d+i-1}{2}\rfloor \choose i-1}(-x)^i\right).\\
     \intertext{Now, peeling off the $i=d$ and $i=d+1$ terms of the second sum which are $(-1)^{\lfloor \frac{d-2}{2}\rfloor}(-x)^d$ and $(-1)^{\lfloor \frac{d-1}{2}\rfloor}(-x)^{d+1}$ respectively,
     and combining the two sums yields}
     &= 1-\left(\left\lfloor \frac{d}{2}\right\rfloor+1\right)x
     -\sum_{i=2}^{d-1} \left((-1)^{\lfloor \frac{i-1}{2}\rfloor} {\lfloor \frac{d+i-1}{2}\rfloor \choose i} + (-1)^{\lfloor \frac{i-2}{2}\rfloor+i} {\lfloor \frac{d+i-1}{2}\rfloor \choose i-1}\right) x^i+(-1)^{\lfloor \frac{d-2}{2}\rfloor+d}x^d \\
     &\qquad + (-1)^{\lfloor \frac{d-1}{2}\rfloor+d+1}x^{d+1}.\\
     \intertext{Using the identity $\left\lfloor \frac{i-1}{2}\right\rfloor \equiv i + \left\lfloor \frac{i-2}{2}\right\rfloor \pmod{2}$, we obtain}
     &= 1-\left(\left\lfloor \frac{d}{2}\right\rfloor+1\right)x
     -\sum_{i=2}^{d-1} (-1)^{\lfloor \frac{i-1}{2}\rfloor}\left( {\lfloor \frac{d+i-1}{2}\rfloor \choose i} + {\lfloor \frac{d+i-1}{2}\rfloor \choose i-1}\right) x^i+(-1)^{\lfloor \frac{d-1}{2}\rfloor}x^d + (-1)^{\lfloor \frac{d}{2} \rfloor} x^{d+1}.
     \intertext{Finally, applying Pascal's identity yields}
          &= 1-\left(\left\lfloor \frac{d}{2}\right\rfloor+1\right)x
     -\sum_{i=2}^{d-1} (-1)^{\lfloor \frac{i-1}{2}\rfloor}{\lfloor \frac{d+i+1}{2}\rfloor \choose i} x^i+(-1)^{\lfloor \frac{d-1}{2}\rfloor}x^d + (-1)^{\lfloor \frac{d}{2} \rfloor} x^{d+1}.
    \end{align*}
We recognize this quantity now as $p_d(x)$, as desired.
\end{proof}

\begin{remark}
One may notice a discrepancy between $A_2(x)$ as given in Theorem \ref{thm:gen_fun_formula} and as given in \cite[Example 6.2]{zbMATH08171225}. This is a result of a difference in indexing: here we index by the length of $\mathbf{d}$ whereas the cited article indexes by the number of tiles. This results in our sequences being offset by one, hence a different numerator in the generating function.
\end{remark}

\begin{example}
\label{ex:determinantofN_5}
Here we carry out a computation of the determinant of the matrix
$N_5(x)$ as defined in the proof of Theorem~\ref{thm:gen_fun_formula}.
We proceed by cofactor expansion along the first column. Our first
term is then given by the value $1$ multiplied by the determinant
of the matrix obtained by deleting the first row and column (highlighted
below in red with a dashed outline). The determinant of this submatrix
can then be computed by cofactor expansion along the bottom row, yielding
the determinant of the submatrix obtained by deleting its last row and column (highlighted below in orange with a solid outline). Note that this submatrix
is precisely $N_3(x)$.
    \[\begin{tikzpicture}
    \draw[fill=red!40!white, draw=red, very thick, dashed] (-1.75,.85) rectangle ++(4.35,-2.1);
    \draw[fill=red!50!yellow!40!white, draw=red!50!yellow!90!black, ultra thick] (-1.7,.8) rectangle ++(3.6,-1.6);
        \node {$\begin{bmatrix}
            1 & -1 & 0 & 0 & 0 & -x\\
            0 & 1 & -1 & 0 & -x & 0\\
            0 & 0 & 1 & -x-1 & 0 & 0\\
            0 & 0 & -x & 1 & -1 & 0\\
            0 & -x & 0 & 0 & 1 & -1\\
            -x & 0 & 0 & 0 & 0 & 1
        \end{bmatrix}$};
    \end{tikzpicture}\]

Continuing our cofactor expansion along the first column of $N_5(x)$, our
only other term is $(-1)^6x$ multiplied by the determinant of the matrix obtained by deleting the first column and last row (highlighted below in cyan with a solid outline).
    \[\begin{tikzpicture}
    \draw[fill=cyan!40!white, draw=cyan!90!black, very thick] (-1.8,1.23) rectangle ++(4.45,-2.1);
        \node {$\begin{bmatrix}
            1 & -1 & 0 & 0 & 0 & -x\\
            0 & 1 & -1 & 0 & -x & 0\\
            0 & 0 & 1 & -x-1 & 0 & 0\\
            0 & 0 & -x & 1 & -1 & 0\\
            0 & -x & 0 & 0 & 1 & -1\\
            -x & 0 & 0 & 0 & 0 & 1
        \end{bmatrix}$};
    \end{tikzpicture}\]

Note that this matrix looks very similar to $N_4(x)$, but it has been transformed in the following ways:
\begin{enumerate}
    \item the order of its rows has been reversed (left multiplication by $J_5$),
    \item the order of its columns has been reversed (right multiplication by $J_5$),
    \item it has been negated, and
    \item $x$ has been replaced with $-x$.
\end{enumerate}
That is, this submatrix is given by \[-J_5N_4(-x)J_5.\] Adding these two terms yields
\begin{align*}
    \det(N_5(x))&=\det(N_3(x))+(-1)^{6}x\det(-J_5N_4(-x)J_5)\\
    &=\det(N_3(x))-x\det(J_5)\det(N_4(-x))\det(J_5)\\
    &=\det(N_3(x))-x\det(N_4(-x)).
\end{align*}
\end{example}

\section{Future work}
\label{sec:futurework}

We began this project considering the degeneration order for representations of general
type $A$ quivers, and we restricted to the case of complexes both because it is a more
tractable case and because of the connections we uncovered with mixed dimer covers.
Our primary interest, then, is in how this complex condition can be relaxed (e.g.
by requiring each \emph{triple} composition be zero rather than any composition) and whether one can naturally
extend the mixed dimer interpretation to these cases.

The module category of $C_n$ is completely described in \cite{Schemes}. That paper also considers the algebra $\widetilde{C_n}$, which is the quotient of the path algebra of a directed cycle modulo all length two paths. The algebras $C_n$ and $\widetilde{C_n}$ share many properties, and it would be interesting to find a dimer interpretation in the setting of $\widetilde{C_n}$. A natural suggestion is that these should relate to \emph{band graphs}, as in \cite{musiker2013bases}. Higher dimer covers of band graphs have recently been considered in \cite{kang2026higherqcontinuedfractionsdimers}.

Another natural question when one has a new indexing set for a set of geometric objects
is whether one can read geometric data from the combinatorics of those
objects. In this paper, we primarily concern ourselves with the degeneration order (how
these geometric objects fit together), but one could ask whether information like the
dimension of the orbit indexed by a mixed dimer cover has combinatorial meaning for that cover.

\section*{Acknowledgements}

E.B. was supported by the German Research Foundation SFB-TRR 358/1 2023 – 491392403.

This material is based on work done while the authors were participating in the Mathematics Research Communities (MRC) 2024
Summer Conference in Java Center, New York, supported by the National Science Foundation under Grant Number DMS 1916439. 
We thank the American Mathematical Society for their support. We also thank Pamela Harris for teaching us about Kostant's partition function.

The authors made use of generative AI in preparing this manuscript; specifically, the matrix factorization as a first step in the proof of Theorem \ref{thm:gen_fun_formula} was suggested by Gemini \cite{gemini}. The remaining arguments herein are due to the authors.

\bibliographystyle{abbrv}
\bibliography{bibliography}

\end{document}